\documentclass{amsart}

\usepackage{bm}
\usepackage{amssymb,amsthm}
\usepackage{thmtools}
\usepackage{mathtools}
\usepackage{bbm} 
\usepackage{hyperref}
\usepackage[capitalise]{cleveref}

\declaretheorem[style=plain, name=Theorem, numberwithin=section]{theorem}
\declaretheorem[style=plain, name=Lemma, sibling=theorem]{lemma}
\declaretheorem[style=plain, name=Proposition, sibling=theorem]{proposition}
\declaretheorem[style=plain, name=Corollary, sibling=theorem]{corollary}

\declaretheorem[style=definition, name=Definition, sibling=theorem]{definition}

\declaretheorem[style=remark, name=Remark, sibling=theorem]{remark}
\declaretheorem[style=remark, name=Notation, sibling=theorem]{notation}

\newcommand{\cantor}{2^{\omega}}
\newcommand{\baire}{\omega^{\omega}}

\newcommand{\bPi}[2]{\bm{\Pi}^{#1}_{#2}}

\newcommand{\lPi}[2]{\Pi^{#1}_{#2}}
\newcommand{\lDelta}[2]{\Delta^{#1}_{#2}}

\newcommand{\upairs}{2^{<\omega} \times \omega^{<\omega}}
\newcommand{\cl}{\mathrm{cl}}

\newcommand{\open}[1]{[#1]^{\prec}}
\newcommand{\laver}{\mathbb{L}}

\newcommand{\concat}[2]{#1^{\smallfrown}#2}
\newcommand{\stem}[1]{\mathrm{stem}(#1)}

\newcommand{\T}{\mathrm{T}}

\newcommand{\ideal}[1]{\mathcal{#1}}
\newcommand{\ock}{\omega^{\mathrm{CK}}_1}

\newcommand{\atr}{\ensuremath{\mathsf{ATR}_0}}

\DeclareMathOperator{\dom}{dom}

\newcommand{\upto}{{\upharpoonright}}

\newcommand{\nl}[2]{\mathrm{NL}^{#1}(#2)}
\newcommand{\stm}[1]{\tau^{\bm{#1}}}
\newcommand{\tree}[1]{T^{\bm{#1}}}
\newcommand{\fun}[1]{f^{\bm{#1}}}
\newcommand{\mindom}[1]{\sigma^{\bm{#1}}}
\newcommand{\bad}[1]{B^{\bm{#1}}}

\newcommand{\cohen}{\mathbb{C}}
\newcommand{\hechler}{\mathbb{H}}

\newcommand{\hechlerb}{\mathbb{H}_{\mathsf{B}}}
\newcommand{\iter}{\mathbb{I}}
\newcommand{\bushy}{\mathbb{O}}
\newcommand{\hyp}{\mathsf{HYP}}
\newcommand{\lhyp}{\leq_{\hyp}}

\newcommand{\SME}[1]{#1\text{-}\mathsf{SME}}
\newcommand{\SNE}[1]{#1\text{-}\mathsf{SNE}}
\newcommand{\domm}[1]{#1\text{-}\mathsf{DOM}}

\newcommand{\wpdom}{\mathsf{Dom}}
\newcommand{\wpesc}{\mathsf{Esc}}
\newcommand{\spill}[1]{\mathsf{Spill}(\mathcal{#1})}
\newcommand{\capture}[1]{\mathsf{Capture}(\mathcal{#1})}

\newcommand{\force}{\Vdash}
\newcommand{\vphi}{\varphi}
\DeclareMathOperator{\rk}{rk}

\newcommand{\seq}[1]{{\left\langle{#1}\right\rangle}}
\newcommand{\smallseq}[1]{\langle{#1}\rangle}
\newcommand{\rooot}{\seq{}}

\newcommand{\converge}{\!\!\downarrow}

\author{Noam Greenberg}
\address{School of Mathematics and Statistics\\
Victoria University of Wellington\\
Wellington\\
New Zealand}
\email{\href{mailto:greenberg@msor.vuw.ac.nz}{greenberg@sms.vuw.ac.nz}}

\author{Gian Marco Osso}
\address
{Dipartimento di Scienze Matematiche, Informatiche e Fisiche\\
	Universit\`a di Udine\\
	33100 Udine\\
	Italy}
\email{\href{mailto:osso.gianmarco@spes.uniud.it}{osso.gianmarco@spes.uniud.it}}

\thanks{The authors would like to thank Jun Le Goh, Alberto Marcone, Joseph Miller, Isabella Scott, Mariya Soskova and Dan Turetsky for useful conversations on the contents of this paper.}
\date{\today}

\title{Forcing and classes of $\hyp$-dominating functions}
\begin{document}

	\maketitle
	
	\begin{abstract}
		We study the question, what computational power is sufficient to perform constructions using either Laver or Hechler forcing. As a result, we obtain a separation between three relativised non-lowness classes that are the computability-theoretic analogues of three of the cardinals in Cichon's diagram. 
	\end{abstract}

\section{Introduction}
	A Turing degree bounds a weakly 1-generic real if and only if it is hyperimmune, that is, computes a function that escapes all computable functions. One direction of this equivalence is immediate: thought of as an element of Baire space, a 1-generic real \emph{is} a function that escapes domination by computable functions. The other direction uses a particular property of codes of the dense sets that we need to meet, namely, they are computably enumerable (c.e.), and there is a uniform enumeration of all c.e.\ sets.  Intuitively, we feel that the first implication is more ``necessary'' than the other. 

	A way of making this intuition precise is examining the situation in other contexts. At the far end of the spectrum, we can consider domination and genericity over a ground model of set theory. Here we have a separation: every Cohen generic function escapes domination by any function in the ground model; however, there are generic extensions containing such escaping functions, in which there are no Cohen generics. Such extensions are given for example by forcing with rational perfect trees (Miller forcing). 

	A natural question is then to measure what computational strength is required to make Miller forcing ``work''. In the argument that such forcing indeed does not add Cohen reals, we need to pass from a given condition to an extension that is more complicated. In particular, we need to ask questions about values of a functional on a given Miller tree~$T$, that can be answered by $T''$ (the double jump of~$T$). This is, in some sense, why we cannot get a separation at the computable level: the collection of computable Miller trees is insufficient to carry out the argument, so we cannot use it to produce a hyperimmune oracle that does not compute a weakly 1-generic. However, the fact that we only need a few jumps to carry out the argument, shows that, for example, there is a function that escapes all arithmetically definable functions, but does not compute an arithmetically generic real, even with the help of an arithmetic oracle. 

	In general, we work over a countable Turing ideal $\ideal{I}$, and use Turing reducibility relativised to $\ideal{I}$:
	\begin{definition} \label{def:I-computability}
		For a Turing ideal~$\ideal{I}$ and $x,y\in \baire$, we write $x\le_{\T(\ideal{I})} y$ if $x\le_{\T}y\oplus z$ for some $z\in \ideal{I}$; we say that $y$ \emph{$\ideal{I}$-computes}~$x$.   
	\end{definition}
	We ask what closure properties of~$\ideal{I}$ are sufficient (and perhaps necessary) so that there is a function~$f$ that is not dominated by any function in~$\ideal{I}$, but does not $\ideal{I}$-compute any $\ideal{I}$-Cohen generic real (a real that is a memeber of every dense open set with code in~$\ideal{I}$). To date, a complete characterisation is not known.\footnote{The current known upper bound is ideals satisfying: for all $X\in \ideal{I}$, there is some $g\in \ideal{I}$ which is diagonally nonrecursive relative to~$X$. The current known lower bound are ideals that fail this property, and also are not closed under finding an escaping function: ideals for which there is some $X\in \ideal{I}$ such that every $g\in \ideal{I}$ is dominated by some $X$-computable function.}

	This question is just one in a broad family of questions. There are several classes of oracles with many equivalent characterisations, with some of the implications more natural than others. For example, the high degrees are those that compute a list of exactly all computable sets. In \cite{DowneyGreenbergMiller}, Downey, Greenberg and Miller show that for a Turing ideal~$\ideal{I}$, there is a list of all the sets in $\ideal{I}$ which does not compute an $\ideal{I}$-dominating function, if and only if $\ideal{I}$ is a model of Weak K\"onig's Lemma. 

	Many of these classes, and implications between them are closely related to \emph{Cichon's diagram}, that describes the relationships between ten cardinal characteristics of the continuum. This connection was first observed by Rupprecht \cite{rupprecht,Rupprecht2010}, and then further studied by Brendle, Brooke-Taylor, Ng and Nies \cite{nies}, Kihara \cite{Kihara18}, Kjos-Hanssen, Stephan and Terwijn \cite{Kjos-HanssenEtAl:Covering}, Monin and Nies \cite{MoninNiesGammaClasses}, and Greenberg, Kuyper and Turetsky \cite{GKT}, among others. The cardinals and the classes of oracles both arise from binary relations that are now known as \emph{Weihrauch problems}. For the binary relations~$A$ that we will consider, we call elements of the domain of~$A$ \emph{instances} of~$A$, and if $a A b$ then we call~$b$ an \emph{$A$-solution} of~$a$. For example, consider the following two Weihrauch problems: 
	\begin{itemize}
		\item The problem $\wpdom$, for which instances are elements $f\in \baire$, and~$g$ is a solution of~$f$ if $g$ dominates~$f$. 
		\item The problem $\capture{\ideal{M}}$, for which instances are elements $f\in \baire$ and a solution of~$f$ is a meagre set containing~$f$. 
	\end{itemize}
	If $A$ is a Weihrauch problem, then:
	\begin{itemize}
		\item We let $|A|$ be the smallest size of a subset~$W$ of the range of~$A$, with the property that every instance of~$A$ has a solution in~$W$; 
		\item For a Turing ideal $\ideal{I}$, we let  $\nl{\ideal{I}}{A}$ be the collection of oracles~$x$ that $\ideal{I}$-compute an instance of~$A$, that has no solution in~$\ideal{I}$. 
	\end{itemize}
	For our examples, $|\wpdom|$ is the \emph{dominating number}, denoted by $\mathfrak{d}$: the smallest size of a set of functions that is cofinal in the partial ordering $(\baire,\le^*)$; $\nl{\ideal{I}}{\wpdom}$ is the collection of oracles that $\ideal{I}$-compute a function escaping all functions in~$\ideal{I}$. $|\capture{\ideal{M}}|$, denoted by $\mathsf{cov}(\ideal{M})$, is the smallest number of meagre sets whose union is the entire real line; $\nl{\ideal{I}}{\capture{\ideal{M}}}$ is the collection of oracles that $\ideal{I}$-compute an \emph{$\ideal{I}$-Cohen} real, one which is not an element of any meagre set with code in~$\ideal{I}$.\footnote{When $\ideal{I}$ is the ideal of computable sets, then an $\ideal{I}$-Cohen real is called a weakly 1-generic real. In this paper, we will only be concerned with jump ideals. When $\ideal{I}$ is a jump ideal, there is no difference between weakly $\ideal{I}$-generic reals and $\ideal{I}$-generic reals. We use the term $\ideal{I}$-Cohen to differentiate with genericity for various other notions of forcing we use.}

	The main tool for understanding relationships between cardinals and classes is the notion of a \emph{morphism} between Weihrauch problems: a morphism from $A$ to~$B$ is a pair of functions (reductions), a function $\psi_{\texttt{inst}}$ from the instances of~$A$ to the instances of~$B$, and a function $\psi_{\texttt{sol}}$ from the solutions of~$B$ to the solutions of~$A$, with the property that for any $A$-instance~$a$, for any $B$-solution~$b$ of $\psi_{\texttt{inst}}(a)$, $\psi_{\texttt{sol}}(b)$ is an $A$-solution of~$a$. The use of morphisms is the following:
	\begin{itemize}
		\item If there is a morphism from~$A$ to~$B$ then $|A|\le |B|$. 
		\item If there is an effective morphism from~$A$ to~$B$ then for any Turing ideal~$\ideal{I}$, $\nl{\ideal{I}}{A}\subseteq \nl{\ideal{I}}{B}$. 
	\end{itemize}
	For example, there is an effective morphism from $\capture{\ideal{M}}$ to $\wpdom$, showing that $\mathsf{cov}(\ideal{M})\le \mathfrak{d}$ is a theorem of ZFC, and that for any Turing ideal~$\ideal{I}$, every $\ideal{I}$-Cohen real computes an $\ideal{I}$-escaping function (and in particular, every weakly 1-generic has hyperimmune degree).\footnote{The terminology in this area is varied; see \cite{GKT} for a survey of various names used for Weihrauch problems, morphisms, and notation for the associated cardinals and classes.}

	The general idea is that if there is an effective morphism from~$A$ to~$B$, then the inclusion $\nl{\ideal{I}}{A}\subseteq \nl{\ideal{I}}{B}$ is ``necessary''; the morphism shows why this inclusion holds in all possible contexts. Often (but not always), when there is no morphism from~$B$ to~$A$ then it is relatively consistent that $|A|<|B|$; and for sufficiently closed ideals $\ideal{I}$, we have $\nl{\ideal{I}}{B}\nsubseteq \nl{\ideal{I}}{A}$. 

	The machinery for showing both relative consistency results and non-inclusion of classes is by finding a notion of forcing that adds an instance of~$B$ with no simple solutions, while not adding complicated instances of~$A$. For example, as mentioned above, Miller forcing adds an escaping function (an instance of $\wpdom$ with no solution in the ground model), while not adding any Cohen reals (instances of $\capture{\ideal{M}}$ without solutions in the ground model). Then, iterating the notion of forcing yields a model in which $|A|<|B|$. And as we saw, if we can effectivise the notion of forcing to work over an ideal~$\ideal{I}$, then we can construct an oracle in $\nl{\ideal{I}}{B}\setminus \nl{\ideal{I}}{A}$. 

	In the current paper, we are concerned with three of the Weihrauch problems and classes that appear in Cichon's diagram.

	\begin{enumerate}
		\item The escaping problem $\wpesc$: an instance is a function~$f$, a solution is a function~$g$ not dominated by~$f$. The cardinal $|\wpesc|$ is $\mathfrak{b}$, the \emph{bounding number}: the smallest size of a collection of functions that is not bounded by a single function. For a Turing ideal $\ideal{I}$, we denote the class $\nl{\ideal{I}}{\wpesc}$ by $\domm{\ideal{I}}$: the collection of oracles that $\ideal{I}$-compute a function that dominates all functions in~$\ideal{I}$. 

		\item The ``spilling problem'' $\spill{\ideal{M}}$ for the meagre ideal: an instance is a meagre set~$A$, a solution is a meagre set $B\nsubseteq A$. The cardinal $|\spill{\ideal{M}}|$ is the \emph{additivity number} of the meagre ideal, $\mathsf{add}(\ideal{M})$: the smallest size of a collection of meagre sets whose union is not meagre. For a Turing ideal $\ideal{I}$, the class $\nl{\ideal{I}}{\spill{\ideal{M}}}$ consists of the oracles that $\ideal{I}$-compute a meagre set which contains all meagre sets in $\ideal{I}$. Such oracles are sometimes known as the \emph{strongly meagre-engulfing} oracles (relative to~$\ideal{I}$), and so we denote this class by $\SME{\ideal{I}}$.

		\item The corresponding problem $\spill{\ideal{N}}$ for the null ideal, which gives the additivity number $\mathsf{add}(\ideal{N})$, and the class $\SNE{\ideal{I}}$ of \emph{strongly null-engulfing} oracles relative to~$\ideal{I}$: those that $\ideal{I}$-compute a null set containing all null sets in~$\ideal{I}$. 
	\end{enumerate}

To make the last two definitions precise, we need to explain what it means to compute a null set and to compute a meagre set. We use the most common definition (as in \cite{rupprecht,nies,GKT}), using \emph{codes} or \emph{names} (the terminology comes from the theory of represented spaces developed by the Weihrauch school of computable analysis, see \cite{Weihrauch:ComputableAnalysis}). A name of an open subset $\mathcal{U}$ of Cantor space is a set $U\subseteq 2^{<\omega}$ of finite binary sequences such that 
\[
	\mathcal{U} = [U]^\prec = \left\{ x\in \cantor \,:\, (\exists \sigma\in U)\,\,\sigma\prec x   \right\}. 
\]
A name of a closed set  $\mathcal{P}\subseteq \baire$ is a name of a tree $T\subseteq \omega^{<\omega}$ such that $\mathcal{P}$ is the set of paths in $T$:
\[
	\mathcal{P} = [T] = \left\{ x\in \baire \,:\,  (\forall \sigma\prec x)\,\,\sigma\in T  \right\}. 
\]
A name of an $F_\sigma$ set $\mathcal{F}$ is a sequence of names of closed sets whose union is $\mathcal{F}$; a name of a $G_\delta$ set $\mathcal{G}$ is a sequence of names of open sets whose intersection is $\mathcal{G}$. We also recall that a name of a real number $r\in \mathbb{R}$ is a \emph{fast Cauchy sequence} converging to~$r$: a sequence of rational numbers $(q_n)$ with $|q_n-r|\le 2^{-n}$ for all~$n$. Armed with these, we can now define: 
\begin{itemize}
	\item A name of a $G_\delta$ null set~$\mathcal{H}$ consists of: a sequence of names of sets $\mathcal{U}_n$ satisfying $\mathcal{H} = \bigcap_n \mathcal{U}_n$; and a sequence of names of the reals $\lambda(\mathcal{U}_n)$, where~$\lambda$ denotes Lebesgue measure on Cantor space. 
	\item A name of an $F_\sigma$ meagre set is simply a name of an $F_\sigma$ set that happens to be meagre. 
\end{itemize}
We are only concerned with names of $G_\delta$ null sets, since every null set is contained in a $G_\delta$ null set. Similarly, every meagre set is contained in an $F_\sigma$ meagre set, so we only name those. We emphasise again that when naming null sets, it is not enough to give a name of a set $\mathcal{H}$ as a $G_\delta$ set; we also need to state the precise measures of the sets $\mathcal{U}_n$ whose intersection is $\mathcal{H}$. In algorithmic randomness, the null sets with computable names are called \emph{Schnorr null} sets.\footnote{We remark, however, that if $x$ is a name of a $G_\delta$ set $\mathcal{H}$ that happens to be null, then $x'$ (the Turing jump of~$x$) computes a name of $\mathcal{H}$ as a null set. Below we will work over jump ideals, so we can ignore the extra requirement to name the measures of the sets $\mathcal{U}_n$.} 

\medskip

The key fact concerning the three problems / classes we listed above, is the existence of two effective morphisms (see \cite{GKT}): 
\begin{itemize}
	\item There is an effective morphism from $\spill{\ideal{N}}$ to $\spill{\ideal{M}}$; 
	\item There is an effective morphism from $\spill{\ideal{M}}$ to $\wpesc$. 
\end{itemize}
Hence, the cardinals satisfy 
\[
	\mathsf{add}(\ideal{N}) \le \mathsf{add}(\ideal{M}) \le \mathfrak{b}; 
\]
and for any Turing ideal~$\ideal{I}$, 
\[
	\SNE{\ideal{I}} \subseteq \SME{\ideal{I}} \subseteq \domm{\ideal{I}}. 
\]

In the other direction, we sometimes get ``accidental'' containments. For example, when $\ideal{I}$ is the ideal of computable sets, then all three classes coincide with the collection of high degrees. We investigate what closure property guarantees separation between classes. Our main theorem is:

\begin{theorem} \label{thm:main}
	Suppose that $\ideal{I}$ is closed under hyperarithmetic reducibility $\lhyp$: for every $X\in \ideal{I}$, every $Y\lhyp X$ (that is, $Y$ which is $\Delta^1_1(X)$) is also in~$\ideal{I}$. Then
	\[
	\SNE{\ideal{I}} \subsetneq \SME{\ideal{I}} \subsetneq \domm{\ideal{I}}. 
\]
\end{theorem}

This result is optimal, in the sense that weaker closure does not guarantee separation. In a follow-up paper, Osso, Miller and Scott \cite{listing} show that for every computable ordinal~$\alpha$, if we let $\ideal{I}_\alpha$ be the collection of reals computable from $\emptyset^{(\beta)}$ for some $\beta<\alpha$, then 
\[
	\SNE{\ideal{I}_\alpha} = \SME{\ideal{I}_\alpha} = \domm{\ideal{I}_\alpha}, 
\]
as they all coincide with the collection of oracles that $\ideal{I}_\alpha$-compute a \emph{list} of $\ideal{I}_\alpha$ --- a list of reals containing exactly the reals in $\ideal{I}_\alpha$. The question of whether we can obtain separation for the ideal of hyperarithmetic sets has been open for a while. 

\smallskip

There are many other reducibilities other than $\le_{\T(\ideal{I})}$ relative to which we can examine separations of associated classes. Some of these were studied by Kihara~\cite{Kihara18}. In particular, for hyperarithmetic reducibility, Kihara observed that the following classes of oracles~$x$ coincide:
\begin{itemize}
	\item Oracles~$x$ such that some $f\lhyp x$ dominates all $\Delta^1_1$ functions; 
	\item Oracles~$x$ such that some meagre set $A\lhyp x$ contains all $\Delta^1_1$ meagre sets; 
	\item Oracles~$x$ such that some null set $B\lhyp x$ contains all $\Delta^1_1$ null sets. 
\end{itemize}
Indeed, a single jump suffices. Suppose that $f$ dominates all~$\Delta^1_1$ functions. Then uniformly in a notation~$a$ for a computable ordinal, $f$ computes some $x=^* \emptyset^{(a)}$ (these iterations of the jump have uniform self-moduli). This implies that~$f'$ computes a \emph{weak list}\footnote{We say that $(x_n)$ is a weak list of $\ideal{I}$ if $\ideal{I} \subseteq \{x_n : n \in \omega\}$. It is known that, for any $\ideal{I}$, if $x$ is a weak list of $\ideal{I}$, then $x \in \SNE{\ideal{I}}$.} of all $\Delta^1_1$ functions, which implies that $f'\in \SNE{\hyp}$. Hence, the witnesses for the separations in \cref{thm:main} have to be constructed delicately. 

\smallskip

Some more evidence for why the proof of \cref{thm:main} must be fairly elaborate concerns the notions of forcing that we use. The standard tools for the analogous separations in set theory are \emph{Laver forcing} and \emph{Hechler forcing}. We will observe that the ``normal'' strength required for these arguments is actually stronger than just closure under relative $\Delta^1_1$; the natural usage of these notions of forcing goes through over models of $\atr$ (arithmetic transfinite recursion). Furthermore, below we will present an argument of Turetsky's that explains why a straightforward use of these notions of forcing cannot be used to prove \cref{thm:main}. 
To get them to work in the weaker setting, we need to employ ideas from a completely different quarter: that of \emph{bushy tree} forcing with \emph{bad sets}, originally introduced by Kumabe (see \cite{KL} and \cite{MK}). We will see that the infinitary version of bushy trees allows us to extend Kumabe's techniques from c.e.\ sets of bad strings to~$\Pi^1_1$ sets, and combine them with Laver and Hechler forcing, and iterations thereof. 

\medskip

In \cref{sec:Laver} we consider Laver forcing, and the separation between $\domm{\ideal{I}}$ and $\SME{\ideal{I}}$. We first explain the ``standard'' argument, and show that it goes through if $\ideal{I}$ models $\atr$, confirming a conjecture from \cite{GKT}. We then explain why this will not suffice under the weaker assumption of closure under $\lDelta{1}{1}$, and explain how to modify it to get one part of \cref{thm:main} (\cref{cor:separating_dom_and_SNE_in_general}). In \cref{sec:Hechler} we discuss Hechler forcing, and explain how to use it to obtain a weaker result: adding a dominating real without adding a random one. This is in preparation for \cref{sec:Hechler_and_Cohen}, in which we consider the separation between $\SME{\ideal{I}}$ and $\SNE{\ideal{I}}$, using a 2-step iteration of Cohen and Hechler forcing.

\section{Laver forcing} \label{sec:Laver}


Recall that for a tree $T\subseteq \omega^{<\omega}$, the \emph{stem} of~$T$ (denoted by $\stem{T}$) is the longest node that is comparable with all strings on~$T$, that is, the shortest \emph{splitting point} of~$T$. 

\begin{definition} \label{def:laver_trees}
  	A \emph{Laver tree} is a tree $T\subseteq \omega^{<\omega}$ such that every $\sigma\in T$ extending the stem has infinitely many children in~$T$. That is, if $\sigma\in T$ and $\sigma\succeq\stem{T}$ then there are infinitely many~$n$ such that $\concat{\sigma}{n}\in T$. 
  \end{definition}

  If $S$ and~$T$ are trees, then we say that $T$ \emph{extends} $S$ (and write $T\leq S$) when $T\subseteq S$. \emph{Laver forcing} $\laver$ is the partial ordering of all Laver trees, ordered by extension. For a Turing ideal $\ideal{I}$, $\laver(\ideal{I})$ is the collection of all Laver trees in~$\ideal{I}$, again ordered by extension. 

\smallskip

Any sufficiently generic filter $G\subset \laver(\ideal{I})$ gives us a real $x_G\in \baire$, namely  $x_G = \bigcup \left\{ \stem{T} \,:\, T\in G   \right\}$. The generic real is dominating:

\begin{proposition} \label{prop:laver_adds_dominating_functions}
	Let $\ideal{I}$ be a Turing ideal.  If $G\subset \mathbb{\laver(\ideal{I})}$ is sufficiently generic, then $x_G$ dominates every function in~$\ideal{I}$. 
\end{proposition}

\begin{proof}
	Suppose that $T\in \laver(\ideal{I})$; let $f\in \ideal{I}$. We let $S\subseteq T$ be the collection of $\sigma\in T$ such that for all $i$ with $|\stem{T}|\le i < |\sigma|$ we have $\sigma(i)\ge f(i)$. Then $S\in \laver({\ideal{I}})$, and $S$ forces that $x_G$ dominates~$f$. 
\end{proof}

Recall that $p\force_{\laver(\ideal{I})} \vphi(x_G)$ if for every sufficiently generic $G\subset \laver(\ideal{I})$, $\vphi(x_G)$ holds. In the proof of \cref{prop:laver_adds_dominating_functions} we actually obtained something stronger: for every $x\in [S]$, $x$ dominates~$f$, not just those $x\in [S]$ that are sufficiently generic. We will say that a condition~$T$ \emph{strongly forces} $\vphi(x_G)$  if $\vphi(x)$ holds for all $x\in [T]$. In general, for each notion of forcing $\mathbb{P}$ that we will use, we will  define, for each condition $p\in \mathbb{P}$, a closed set $[p]\subseteq \baire$,  and we will say that $p$ \emph{strongly forces} $\vphi$  if for all $x\in [p]$, $\vphi(x)$ holds. For a general framework for notions of forcing with associated closed sets see for example~\cite{CGM}.

\subsection{The Laver property and Cohen generics} 
The Laver property is a well-known property of certain forcing notions (among which, of course, Laver forcing). A forcing notion with the Laver property does not add any Cohen reals. We give a definition adapted to our context of forcing over a countable ideal $\ideal{I}$, and prove that it implies that a generic adds no $\ideal{I}$-Cohen real.

	\begin{definition} \label{def:traces}
		A \emph{trace} of a function~$f$ is a sequence $(F_n)$ of finite subsets of~$\omega$ such that for all~$n$, 
		\begin{itemize}
			\item $f(n)\in F_n$; and
			\item $|F_n|\le 2^n$. 
		\end{itemize}
	\end{definition}
	
	Traces are often called \emph{slaloms}. The bound $|F_n|\le 2^n$ is arbitrary; for our purposes, it is important that we have a fixed, computable bound. 

	\begin{definition} \label{def:I[x]}
		Let $\ideal{I}$ be a Turing ideal, and let $x\in \baire$. We let 
		\[
			\ideal{I}[x] = \left\{ y\in \baire  \,:\,  y\le_{\T(\ideal{I})} x \right\}.
		\]
	\end{definition}

	We say that a function is \emph{$\ideal{I}$-dominated} if it is dominated by some function in~$\ideal{I}$. 

	\begin{definition} \label{def:Laver_property}
		Let $\ideal{I}$ be a countable Turing ideal; let $\mathbb{P}$ be a notion of forcing. We say that $\mathbb{P}$ has the \emph{Laver property} over $\ideal{I}$ if for all sufficiently generic $G\subset \mathbb{P}$, every $g\in \ideal{I}[x_G]$ that is $\ideal{I}$-dominated has a trace in~$\ideal{I}$.
	\end{definition}

	\begin{lemma} \label{lem:Laver_property_and_generics}
		Let $\ideal{I}$ be a countable Turing ideal. Suppose that a notion of forcing $\mathbb{P}$ has the Laver property over $\ideal{I}$. Then for sufficiently generic $G\subset \mathbb{P}$, no $f\in \ideal{I}[x_G]$ is  $\ideal{I}$-Cohen. 
	\end{lemma}
	
	\begin{proof}
		Let $f\in \ideal{I}[{x_G}]$ be an element of Cantor space~$\cantor$; we show how to construct a nowhere dense closed set in~$\ideal{I}$ that contains~$f$. 

		For $n<\omega$ let $\hat f(n) = f\upto 2^n$. Then $\hat f\in \ideal{I}[x_G]$ is dominated by the computable function $n\mapsto 2^{2^n}$, in particular, it is $\ideal{I}$-dominated. Hence, let $(F_n)\in \ideal{I}$ be a trace of~$\hat f$. We may assume that for all~$n$, every $\sigma\in F_n$ has length $2^n$. Let 
		\[
			\mathcal{P} = \left\{ x\in \cantor \,:\, (\forall n)\,\,x\upto n\in F_n    \right\}. 
		\]
		Then $\mathcal{P}$ is closed, is in $\ideal{I}$ (as it is computable from $(F_n)$), contains $f$, and is nowhere dense. To see the latter, we use the fact that $2^n < 2^{2^n}$; for each~$n$, at most $2^n$ strings of length $2^n$ have extensions in~$\mathcal{P}$. 
	\end{proof}

	Our main interest in Cohen generics relies on the fact that there is an effective morphism from $\spill{\ideal{M}}$ to the Weihrauch problem $\capture{\ideal{M}}$ mentioned in the introduction (see \cite[Proposition 3.4]{GKT}). Hence, for any ideal~$\ideal{I}$, every $x\in \SME{\ideal{I}}$ computes an $\ideal{I}$-Cohen real. We conclude:

	\begin{corollary} \label{cor:Laver_property_and_SME}
		Let $\ideal{I}$ be a countable Turing ideal; suppose that a notion of forcing $\mathbb{P}$ has the Laver property over $\ideal{I}$. Then for sufficiently generic $G\subset \mathbb{P}$, $x_G\notin \SME{\ideal{I}}$. 
	\end{corollary}

	\subsection{Fusion of Laver trees}

	The fusion of a sequence of trees was first introduced by Sacks. The notion applies to Laver trees as well.

	\begin{definition} \label{def:P_i_of_a_Laver_tree}
		Let $T$ be a Laver tree.
		We define a family of subsets of $T$ recursively as follows:
		\begin{itemize}
			\item $P_0(T) = \{\stem{T}\}$. 
			\item For $i\ge 0$, $P_{i+1}(T)$ is obtained from $P_i(T)$ by adding the left-most children of elements of $P_i(T)$ that are not yet in $P_i(T)$. That is, 
			\[
				P_{i+1}(T) = P_i(T)\cup \left\{ \concat{\sigma}n \,:\,  \sigma\in P_i(T) \text{ and $n$ is least with } \concat{\sigma}{n}\notin P_i(T) \right\}. 
			\]
		\end{itemize}

		For $i\ge 0$ and Laver trees~$T$ and~$S$, we write $T\le_i S$ if $T\subseteq S$ and $P_i(T) = P_i(S)$. 
	\end{definition}

	Like fusion with Sacks trees, the idea is to present a tree~$T$ as an increasing union of finite subtrees (ignoring the part below the stem). That is, 
	 \[
	 T = \left\{ \sigma  \,:\,  \sigma\prec \stem{T} \right\} \cup \bigcup_i P_i(T). 
	 \]
	 The particular definition is chosen so that for all~$T$ and~$i$, 
	\[
		|P_i(T)| = 2^i.
	\]
	Since at each stage we add a child to each existing node, if $j>i$ and $\sigma\in P_i(T)$ then $\sigma$ has $j-i$ many children in $P_j(T)\setminus P_i(T)$.

	\begin{lemma}[Fusion for Laver trees]\label{lem:fusion}
	Suppose that $T_0, T_1, \dots$ are Laver trees, and that 
	\[
		T_0 \ge_0 T_1 \ge_1 T_2 \ge_2 T_3 \ge_3\cdots.
	\]
	Let $T= \bigcap_{i \in \omega} T_i$. Then~$T$ is a Laver tree, and for all~$i$, $T\le_i T_i$. 
	\end{lemma}

	A sequence $(T_i)$ as in the lemma can be viewed as the result of a recursive process of thinning of~$T_0$: at stage~$i$ we have~$T_i$. To define $T_{i+1}$, we first declare that the strings in $P_i(T_i)$ are protected, and cannot be removed (now or at any future stage). For each node~$\sigma$ we can nonetheless remove infinitely many children when passing from $T_i$ to $T_{i+1}$, as long as we keep infinitely many. 

	\begin{proof}
		Since $\le_{i+1}$ implies $\le_i$, and each $\le_i$ is transitive, for all $j\ge i$ we have $T_j \leq_{i} T_i$, so $P_i(T_i) \subseteq T$. In particular, all the trees~$T_i$ have the same stem, call it~$\sigma$, and
		\[
			T = \left\{ \rho  \,:\,  \rho\prec \sigma \right\} \cup \bigcup_i P_i(T_i). 
		\]
		To see the containment from left to right, since $T\subseteq T_0$, every string in~$T$ is comparable with~$\sigma$. By induction on the length of $\tau\in T$ extending~$\sigma$, we show that $\tau\in P_i(T_i)$ for some~$i$. We have $\sigma\in P_0(T_0)$. If $\tau\succ\sigma$, let $\bar\tau$ be the parent of~$\tau$, and let~$i$ such that $\bar \tau\in P_i(T_i)$. Suppose that~$\tau$ is the $k^{\textup{th}}$ child of $\bar \tau$ in~$T_i$. As $\tau\in T_{i+k}$ and $P_i(T_{i+k}) = P_i(T_i)$, we have $\tau\in P_{i+k}(T_{i+k})$. The same argument shows that $\bar \tau$ has infinitely many children on~$T$, so $T$ is a Laver tree, and $P_i(T) = P_i(T_i)$. 
	\end{proof}

	A sequence $(T_i)$ as above is called a \emph{fusion sequence}, and the resulting tree~$T$ is called the {fusion} of the sequence.	For the purposes of constructing fusion sequences below, we use the following partition of a given Laver tree. For the following definition and below, for a tree~$T$ and a node $\tau\in T$, we let $T\upto\tau$ denote the full subtree issuing from~$\tau$, that is, the collection of all $\sigma\in T$ comparable with~$\tau$. If $T$ is a Laver tree, then so is $T\upto{\tau}$; if $\tau\succeq \stem{T}$ then the stem of $T\upto{\tau}$ is~$\tau$. 

	\begin{definition} \label{def:fusion_partition}
		Let $T$ be a Laver tree, and let $i<\omega$. For $\tau \in P_i(T)$ we let $T^i_\tau$ be the result of removing from $T\upto\tau$ all proper extensions of~$\tau$ in $P_i(T)$ (and their extensions). 
	\end{definition}
	
	Since we are only removing finitely many children of~$\tau$, each $T^i_\tau$ is a Laver tree, with $\stem{T^i_\tau} = \tau$. The collection of closed sets $[T^i_\tau]$ for $\tau\in P_i(T)$ is a finite partition of $[T]$.

	\subsection{$\omega$-bushy trees and a notion of largeness}\label{subsect:big}

	As mentioned in the introduction, we will use an infinitary version of notions arising from forcing with bushy trees (see \cite{MK} for a survey). For the following definition, we weaken the notion of ``tree'', to allow a root that is not necessarily the empty sequence. More precisely, a \emph{tree above $\sigma$} is a set~$T$ of nodes $\tau\succeq\sigma$ such that for all $\tau\in T$, for all~$\rho$ with $\sigma\preceq \rho \preceq \tau$, we have $\rho\in T$.

	\begin{definition} \label{def:omega-bushy}
		Let $\sigma\in \omega^{<\omega}$. A tree~$T$ above~$\sigma$ 
		is \emph{$\omega$-bushy} above $\sigma$ if for every $\tau\in T$, either $\tau$ is a leaf of~$T$, or $\tau$ has infinitely many children $\concat{\tau}{n}$ in~$T$.
	\end{definition}
	
	Note the difference with Laver trees: the latter are not allowed to have leaves. Our interest will be with well-founded $\omega$-bushy trees.

	\begin{definition} \label{def:omega-big}
		Let $B\subseteq \omega^{<\omega}$. We say that $B$ is \emph{$\omega$-big} above $\sigma$ if there is a well-founded tree~$T$, $\omega$-bushy above~$\sigma$, such that every leaf of~$T$ is in~$B$. 
	\end{definition}

	Central to our investigations is a characterisation of bigness ``from below'', using a transfinite ranking procedure. Let $B\subseteq \omega^{<\omega}$. By recursion on ordinals $\alpha$, we define the relation ``$\rk_B(\sigma)\le  \alpha$'' for $\sigma\in \omega^{<\omega}$. If this has been defined for all $\beta<\alpha$, then we write $\rk_B(\sigma)<\alpha$ if $\rk_B(\sigma)\le \beta$ for some $\beta <\alpha$. 
	\begin{itemize}
		\item $\rk_B(\sigma)\le 0$ if  $\sigma\in B$. 
		\item For $\alpha>0$, $\rk_B(\sigma)\le \alpha$ if there are infinitely many~$n$ with $\rk_B(\concat{\sigma}n)<\alpha$.
	\end{itemize}
	If $\rk_B(\sigma)\le \alpha$ for some~$\alpha$, then we let $\rk_B(\sigma)$ denote the least such~$\alpha$, and we write $\rk_B(\sigma)<\infty$.  Otherwise we write $\rk_B(\sigma)=\infty$. (Observe that if $\rk_B(\sigma)\le \alpha$ then $\rk_B(\sigma)\le \beta$ for all $\beta\ge \alpha$, justifying this notation.)

	\begin{lemma} \label{lem:equivalence_of_ranked_and_bigness}
		For any $B\subseteq \omega^{<\omega}$ and any $\sigma\in \omega^{<\omega}$, $\rk_B(\sigma)<\infty$ if and only if $B$ is $\omega$-big above~$\sigma$. 
	\end{lemma}

	\begin{proof}
		Suppose that $B$ is $\omega$-big above $\sigma$; let $T$ be a well-founded, $\omega$-bushy tree above~$\sigma$ witnessing this. By induction on the tree-rank $\rk_T(\tau)$ of $\tau\in T$ we can see that $\rk_B(\tau)\le \rk_T(\tau)$; it follows that $\rk_B(\sigma)<\infty$. 

		In the other direction, suppose that $\rk_B(\sigma)<\infty$. Let~$T$ be the collection of all sequences $\tau\succeq \sigma$ such that for all $\rho,\rho'$ with $\sigma \preceq \rho\prec \rho' \preceq \tau$ we have $\rk_B(\rho')< \rk_B(\rho)$. The definition of~$T$ implies that it is well-founded, indeed its rank is bounded by $\rk_B(\sigma)$. If $\tau\in T$ and $\rk_B(\tau)>0$ then there are infinitely many~$n$ with $\rk_B(\concat{\tau}n)<\rk_B(\tau)$, showing that $\tau$ has infinitely many children on~$T$. If $\tau\in T$ and $\rk_B(\tau)=0$ then $\tau$ must be a leaf of~$T$, and it is an element of~$B$. Hence, $T$ is $\omega$-bushy above~$\sigma$, and witnesses that~$B$ is $\omega$-big above~$\sigma$. 
	\end{proof}
	
	We obtain analogues of two fundamental facts about $n$-big sets: the concatenation property \cite[Lem.\:2.6]{MK} and the smallness preservation property \cite[Lem.\:2.7]{MK}:

	\begin{lemma} \label{lem:concatenation_and_preservation_of_smallness}
		Let $A,B,C\subseteq \omega^{<\omega}$. 
		\begin{enumerate}
			\item Suppose that $B$ is $\omega$-big above~$\sigma$, and that for all $\tau\in B$, $C$ is $\omega$-big above~$\tau$.  Then~$C$ is $\omega$-big above~$\sigma$. 

			\item If $A\cup B$ is $\omega$-big above~$\sigma$, then either~$A$ or~$B$ are $\omega$-big above~$\sigma$. 
		\end{enumerate}
	\end{lemma}
	
	\begin{proof}
		(1): Let~$T$ witness that $B$ is $\omega$-big above~$\sigma$, and for $\tau\in B$, let $T_\tau$ witness that~$C$ is $\omega$-big above~$\tau$. Let~$S$ be the union of~$T$ and the trees $T_\tau$ for leaves~$\tau$ of~$T$. Then~$S$ witnesses that~$C$ is $\omega$-big above~$\sigma$. 

		(2): By induction on $\rk_{A\cup B}(\tau)$, we observe that $\rk_{A\cup B}(\tau)= \min \{\rk_A(\tau), \rk_B(\tau)\}$. It follows that $\rk_A(\sigma)<\infty$ or $\rk_B(\sigma)<\infty$. 
	\end{proof}
	
	We can rephrase these properties in the language of closure operators.

	\begin{definition} \label{def:closure_of_B}
		For a set $B\subseteq \omega^{<\omega}$, we let $\cl(B)$ denote the collection of~$\sigma$ such that~$B$ is $\omega$-big above~$\sigma$. 
	\end{definition}

	\Cref{lem:concatenation_and_preservation_of_smallness} implies:

	\begin{proposition} \label{prop:closure_operator}
		The operator~$\cl$ satisfies:
		\begin{enumerate}
			\item $B\subseteq \cl(B)$; 
			\item $\cl(\cl(B)) = \cl(B)$; and
			\item $\cl(A\cup B) = \cl(A)\cup \cl(B)$. 
		\end{enumerate}
	\end{proposition}

	\begin{definition} \label{def:omega-closed}
		We say that a set $B$ is \emph{$\omega$-closed} if for all~$\sigma$, if for infinitely many~$n$, $\concat{\sigma}n\in B$, then $\sigma\in B$. 
	\end{definition}

	\begin{lemma} \label{lem:characterisation_of_omega-closed}
		A set $B\subseteq \omega^{<\omega}$ is $\omega$-closed if and only if $B = \cl(B)$. 
	\end{lemma}

	\begin{proof}
		Suppose that $B$ is $\omega$-closed. Then by induction on $\alpha\ge 0$, we see that for all $\sigma\notin B$, $\rk_B(\sigma)\nleq \alpha$; so $\rk_B(\sigma)=\infty$ for all $\sigma\notin B$; so $\cl(B)=B$. In the other direction, suppose that for infinitely many~$n$, $\concat{\sigma}n\in B$. Then $\rk_B(\sigma)\le 1$, so $\sigma\in \cl(B)$. If $B = \cl(B)$, then $\sigma\in B$. 
	\end{proof}

\begin{lemma} \label{rmk:closure_stays_inside_tree}
	Let $T\subseteq \omega^{<\omega}$ be a tree, and let $B\subseteq T$.
	\begin{enumerate}
		\item $\cl(B)\subseteq T$. 
		\item If $B$ is $\omega$-closed, then so is its upwards closure in~$T$. 
		\item If $A\subseteq \omega^{<\omega}$ is $\omega$-closed, then so is $A\cap T$. 
	\end{enumerate}
\end{lemma}

\begin{proof}
	(1): Let $\sigma\in \cl(B)$; let $R$ be an $\omega$-bushy tree above~$\sigma$ witnessing this. Since~$R$ is well-founded and its leaves are in~$B$, and hence in~$T$, $R\subseteq T$, so $\sigma\in T$. 

	(2): Let~$C$ be the upward closure of~$B$ in~$T$. Let $\sigma\in \omega^{<\omega}$, and suppose that for infinitely many~$n$, $\concat{\sigma}n\in C$; we show that $\sigma\in C$. Note that since $C\subseteq T$, we have $\sigma\in T$. There are two possibilities. If for all such~$n$, $\concat{\sigma}n\in B$, then $\sigma\in B$, so $\sigma\in C$. Otherwise, $\sigma$ itself extends some string in~$B$, so $\sigma\in C$. 

	(3): If $\sigma$ has infinitely many children in $A\cap T$, then $\sigma\in T$ (since~$T$ is a tree) and $\sigma\in A$ (since~$A$ is $\omega$-closed). 
\end{proof}

	We turn to calculations of complexity. 

	\begin{proposition} \label{prop:closure_of_Pi11_sets}
		Suppose that $B\subseteq \omega^{<\omega}$ is $\Pi^1_1$. 
		\begin{enumerate}
			\item For all $\sigma$, if $\rk_B(\sigma)<\infty$ then $\rk_B(\sigma)< \ock$.
			\item If $B$ is $\omega$-big above~$\sigma$, then there is a $\Delta^1_1$ tree~$T$, $\omega$-bushy above~$\sigma$, witnessing this. 
			\item $\cl(B)$ is $\Pi^1_1$. 
		\end{enumerate}
	\end{proposition}
	
	 \begin{proof}
	 	Let $(B_s)_{s<\ock}$ be an $\ock$-computable enumeration of~$B$. For each $\alpha<\ock$ and $s<\ock$, the collection of~$\sigma$ such that $\rk_{B_s}(\sigma)= \alpha$ is $\ock$-computable, uniformly in~$\alpha$ and~$s$. If $s<t$ then $B_s\subseteq B_t$, and so for all~$\sigma$, $\rk_{B_t}(\sigma)\le \rk_{B_s}(\sigma)$. By induction on $\alpha<\ock$, we see that if $\rk_B(\sigma)=\alpha$ then from some computable stage $s$ we have $\rk_{B_s}(\sigma) =\alpha$. This is immediate for $\alpha=0$ (as $B = \bigcup_s B_s$). Suppose that $\alpha>0$ and that this is known for all $\beta<\alpha$. Suppose that $\rk_B(\sigma)=\alpha$. Then 
	 	\[
	 		A = \{ n\,:\, \rk_B(\concat{\sigma}{n})<\alpha\}
	 	\]
	 	is infinite and $\ock$-c.e.\ (that is, $\Pi^1_1$). Since $\ock$ is admissible, $A$ has an infinite $\Delta^1_1$ subset~$C$. Again since $\ock$ is admissible, and by the induction assumption, there is some~$s$ such that for all $n\in C$, $\rk_{B_s}(\concat{\sigma}{n}) = \rk_B(\concat{\sigma}{n})$. Then $\rk_{B_s}(\sigma)=\alpha $ (and so $\rk_{B_t}(\sigma) = \alpha$ for all computable $t\ge s$). 

	 	There is no $\sigma$ with $\rk_B(\sigma)=\ock$: suppose that $\rk_B(\sigma)\le \ock$. Then $\{ n\,:\, \rk_B(\concat{\sigma}{n})<\ock\}$ is infinite and as we just observed, is~$\Pi^1_1$. Let $D$ be an infinite, $\Delta^1_1$ set of such $n$'s. By admissibility of $\ock$, the ranks $\rk_B(\concat{\sigma}{n})$ for $n\in D$ are all bounded below $\ock$; such a bound also bounds $\rk_B(\sigma)$. Now (1) follows by induction on the ordinals $\ge \ock$. 

	 	(2) follows from the analysis above: suppose that $\rk_B(\sigma)<\infty$. Then $\alpha = \rk_B(\sigma)$ is computable; let $s<\ock$ such that $\rk_{B_s}(\sigma) = \alpha$. Since $B_s$ is $\Delta^1_1$, the definition of the $\omega$-bushy tree in \cref{lem:equivalence_of_ranked_and_bigness} gives a $\Delta^1_1$ tree. 

	 	Finally, (3) follows from (1) and the analysis above. 
	 \end{proof}

	 Note that for (2) of \cref{prop:closure_of_Pi11_sets}, we cannot require that the leaves of~$T$ are \emph{minimal} elements of~$B$. They can be minimal elements of some~$B_s$, but shorter strings may later appear in~$B$.

	 \medskip

	To conclude this subsection, we note that $\omega$-bigness is $\lPi{1}{1}$-complete, both in the sense of $m$-reduction (for subsets of $\omega$) and in the sense of Wadge reduction (for subsets of $\cantor$ or of~$\baire$). The proof of this fact, due to Marcone, is based on the concept of \emph{smooth trees} (present in \cite{marcone} in the form of smooth barriers; see also \cite[Exercises VI.1.8 and VI.1.9]{simpson}).
	
	\begin{proposition}[Marcone]
		The set 
		 \[
		 O=\{B \subseteq \omega^{<\omega} \,:\,  B \text{ is $\omega$-big above }\rooot\} 
		 \]
		 is $\Pi^1_1$-complete, via a computable reduction. 
	\end{proposition}

	The fact that there is a computable reduction from any $\Pi^1_1$ subset of Baire space to~$O$ gives the ``type 1'' computable analogue of this result: the set of indices of computable $B \subseteq \omega^{<\omega}$ that are $\omega$-big above $\rooot$ is $\lPi{1}{1}$-complete for subsets of~$\omega$ (using many-one reducibility). 

	\begin{proof}
		Relativising \cref{prop:closure_of_Pi11_sets} to an oracle, uniformly, shows that $O$ is $\Pi^1_1$. For $\Pi^1_1$-hardness, we exhibit a computable Wadge reduction from the set of wellfounded trees to~$O$. Given any tree $T \subseteq \omega^{<\omega}$, we define the tree 
		 \[
		 T^+=\{\sigma \in \omega^{<\omega}: \exists \tau \in T \,\, (|\tau|=|\sigma| \land \tau \leq \sigma)\},
		 \]
		 where $\tau\le \sigma$ means $\tau(i)\le \sigma(i)$ for all $i\le |\tau|$. Note that $T^+$ is computable from~$T$, uniformly. We define a computable function by mapping a tree~$T$ to $\omega^{<\omega} \setminus T^+$.
		
		If $T$ is well-founded, then so is $T^+$ (see, for example, \cite[Exercise VI.1.8]{simpson}; if $f\in [T^+]$ then use the fact that the collection of $h\le f$ is compact, and that~$T$ must be infinite). This implies that $\omega^{<\omega} \setminus T^+$ is $\omega$-big above~$\rooot$: let $A$ be the collection of minimal strings in $\omega^{<\omega} \setminus T^+$; let~$S$ be the downward closure of~$A$. Then~$S$ is well-founded and $\omega$-bushy above~$\rooot$, indeed, if $\sigma\in S$ is not in~$A$ then $\concat{\sigma}n\in S$ for all~$n$. 

		Suppose, on the other hand, that $T$ is ill-founded; let $f\in [T]$. Then every $\tau\ge f$ is in~$T^+$. By choosing large children, this shows that for any well-founded tree~$S$ that is $\omega$-bushy above~$\rooot$, some leaf of~$S$ is in $T^+$, hence $\omega^{<\omega}\setminus T^+$ is not $\omega$-big above~$\rooot$. 
	\end{proof}

	\begin{remark}\label{rem:fusionleaves}
		The notions defined above for fusions of Laver trees (the sets $P_i(T)$, partial ordering $\le_i$, etc.)  can be extended from Laver trees to trees that are $\omega$-bushy above their stem. That is, we can allow leaves; in the definition of $P_{i+1}(T)$ we of course omit adding a child $\concat{\sigma}n$ if such a child does not exist. The fusion~$T$ of a sequence $(T_i)$ will be $\omega$-bushy above the stem; every leaf of~$T$ is a leaf of some $T_i$. When forcing with $f$-bushy trees with bad sets of arbitrary complexity (rather than just c.e.\ ones), we must allow the trees to have some leaves (just not a big set of leaves). In this paper, we only force with $\Pi^1_1$ bad sets (the analogue of~c.e.\ in this setting), so we can restrict ourselves to Laver trees. 
	\end{remark}

\subsection{Laver forcing over $\atr$}

As mentioned in the introduction, before we show that $\SME{\ideal{I}}\subsetneq \domm{\ideal{I}}$ for any $\lhyp$-closed countable ideal~$\ideal{I}$, we first consider the simpler situation, when~$\ideal{I}$ is a countable $\omega$-model of~$\atr$. We will show that this gives enough closure to carry out the forcing argument from set theory. Recall that $\atr$ is the subsystem of second order arithmetic, defined by the axiom of {arithmetic transfinite recursion}. It allows for iterations of the Turing jump along arbitrary linear orderings in the model that the model thinks are well-orderings. This is in contrast with ideals closed under $\lhyp$, where iterations of the Turing jump are provided only for genuine well-orderings in~$\ideal{I}$. 

The fundamental result about Laver forcing that requires $\atr$, is the following.
	
\begin{proposition}\label{prop:ATR:Laver:fund}
Assume $\atr$. Let $X$ be a set, and let $\mathcal{U}$ be an open set with an $X$-computable name. Let $T$ be a Laver tree that is $\Delta^1_1(X)$. Then there is some Laver tree $S\le_0 T$ such that either:
\begin{itemize}
	\item $S\in \Delta^1_1(X)$ and $[S]\subseteq \+U$; or
	\item $[S]\subseteq \+U^{\complement}$. 
\end{itemize}
\end{proposition}

(Recall that $S\le_0 T$ means that $S\subseteq T$ and $\stem{S} = \stem{T}$). 
	
\begin{proof}
	Let $B\le_\T X$ be a name of $\mathcal{U}$: an upwards-closed subset of $\omega^{<\omega}$ with $\mathcal{U} = [B]^\prec$. Let $T$ be an $X$-computable Laver tree; let $C = B\cap T$.  

	The important fact is that $\atr$ proves $\Sigma^1_1(X)$ bounding: if $A$ is a $\Sigma^1_1(X)$ set of $X$-computable indices for $X$-computable well-orderings, then there is some $X$-computable well-ordering which is longer than any of the orderings with indices in~$A$. This is because $\atr$ proves that the collection of $X$-computable indices of well-orderings is not $\Sigma^1_1(X)$ (\cite[Lem.\:VIII.3.4]{simpson}), and that any two well-orderings are comparable. 

	For any $X$-computable well-ordering~$L$, we can define a ranking $\rk^L_C$ along~$L$, in the same way that we defined $\rk_C$ above; for all $\gamma\in L$, the collection of $\sigma$ satisfying $\rk^L_C(\sigma)\le \gamma$ is a set, uniformly in~$\gamma$. $\atr$ proves existence and uniqueness of such ranking. Further, $\atr$ proves that sets of strings with $\rk^L_C(\sigma)\le \gamma$ are $\Delta^1_1(X)$, uniformly in~$\gamma$ (and an index for~$L$). 

	For each $X$-computable~$L$ there are two possibilities: either $\stem{T}$ is ranked, or the elements of~$L$ ``run out'' before we can reach $\stem{T}$, in which case we write $\rk^L_{C}(\stem{T}) = \infty$. 

	If there is some $X$-computable~$L$ such that $\rk_C^L(\stem{T})<\infty$, then just as in the proof of \cref{lem:equivalence_of_ranked_and_bigness}, the collection~$R$ of sequences $\sigma\succeq \stem{T}$ which have  strictly decreasing $\rk^L_{C}$-ranks is $\omega$-bushy above~$\stem{T}$, is well-founded, has leaves in~$C$, and is $\Delta^1_1(X)$. We let~$S$ be the result of adding to each leaf of~$R$ all extensions in~$T$; then $S\le_0 T$ is $\Delta^1_1(X)$ and $[S]\subseteq \+U$. 

	Suppose that for every $X$-computable well-ordering~$L$ we have $\rk^L_C(\stem{T})= \infty$. The property ``there is a $C$-ranking along~$L$ that does not reach $\stem{T}$'' is $\Sigma^1_1(X)$, and so by the $\Sigma^1_1(X)$ bounding principle, there is some $X$-computable linear ordering~$L$ that is ill-founded, and some ranking function $\rk_C^L$ satisfying the definition of $C$-ranking, such that $\rk^L_C(\stem{T})= \infty$. 

	Let $(\gamma_n)$ be an infinite descending sequence in~$L$. For every $\sigma\in T$ extending $\stem{T}$, if $\rk^L_C(\sigma)>\gamma_n$ (meaning that it is not the case that $\rk^L_C(\sigma)\le \gamma_n$) then for almost all~$k$ such that $\concat{\sigma}k\in T$ we have $\rk^L_{C}(\concat{\sigma}k)>\gamma_{n+1}$. 

	We thus define $S\le_0 T$ to consist of all $\sigma\in T$ such that for all~$\rho$ with $\stem{T}\preceq \rho\preceq \sigma$ we have $\rk^L_{B\cap T}(\rho)> \gamma_{|\rho|}$. Then~$S$ is a Laver tree, and since $\rk^L_{C}(\sigma)>0$ for all $\sigma\in S$, we have $S\cap C = \emptyset$, so $[S]\cap \mathcal{U} = \emptyset$. 

	Note that in this case we are not guaranteed that $S$ is $\Delta^1_1(X)$; the linear ordering~$L$ is $X$-computable, but the descending sequence $(\gamma_n)$ may fail to be $\Delta^1_1(X)$. 
\end{proof}
	
\begin{remark} 
The proof of \cref{prop:ATR:Laver:fund} actually yields a slightly stronger result relating Laver and Hechler trees; see \cref{rmk:Hechler_in_atr}. The second author and Marcone have shown in \cite{laverpartition} that this strengthening is actually equivalent to $\atr$. 	
\end{remark}

We will apply \cref{prop:ATR:Laver:fund} inside an $\omega$-model $\ideal{I}$ of $\atr$. Several of the notions involved are \emph{not} absolute between $\ideal{I}$ and~$V$: for example, we have $S\in \Delta^1_1(X)$ \emph{in the sense of~$\ideal{I}$}, meaning $S$ is computable from some transfinite iteration of the Turing jump of~$X$, along a linear ordering that $\ideal{I}$ thinks is well-founded. Further, $\ideal{I}\models [S]\subseteq \+U$ only means $[S]\cap \ideal{I}\subseteq \+U$, and does not imply $[S]\subseteq \+U$. In the proof, the $\omega$-bushy tree~$R$ may be ill-founded, just not contain infinite paths in~$\ideal{I}$. On the other hand, in the other case, we constructed a tree~$S$ with $S\cap B = \emptyset$, which means that $[S]\cap \+U = \emptyset$, not only in the sense of~$\ideal{I}$.

We observe, however, that if $[S]\cap \ideal{I}\subseteq \mathcal{U}$, then $S\force_{\laver(\ideal{I})} x_G\in \mathcal{U}$ (even though it may not strongly force this fact): For all $S'\subseteq S$ in $\laver(\ideal{I})$ we have $[S']\cap \mathcal{U}\ne \emptyset$ (as $[S']\cap \ideal{I}\ne \emptyset$), so we can extend $\stem{S'}$ to some string that is contained in~$\mathcal{U}$, and take the full subtree. The next lemma lifts this to the next level in the Borel hierarchy. In computability theory, it is often called ``forcing totality or divergence'', since $\Pi^0_2$ sets (effectively $G_\delta$ sets) are the domains of Turing functionals (partial computable functions on Baire space or Cantor space). 

	\begin{lemma} \label{lem:atr:forcing_totality_or_divergence}
		Suppose that $\ideal{I}$ is a countable $\omega$-model of $\atr$. Let $\mathcal{H}$ be an $\ideal{I}$-computable $\bPi{0}{2}$ set (a $G_\delta$ set with a name in~$\ideal{I}$), and let $T\in \laver({\ideal{I}})$. There is some $S\le T$ in $\laver({\ideal{I}})$ that decides $x_G\in \mathcal{H}$, indeed, that either forces $x_G\in \mathcal{H}$ or strongly forces $x_G\notin \mathcal{H}$. 
	\end{lemma}

	\begin{proof}
		Write $\mathcal{H} = \bigcap_n \mathcal{U}_n$, with $(\mathcal{U}_n)\in \ideal{I}$. If there is some~$n$ and some $S\le T$ in $\laver({\ideal{I}})$ such that $S\force x_G\notin \mathcal{U}_n$ then we are done. Suppose otherwise. Fix some $X\in \ideal{I}$ that computes~$T$ and a name of~$\+H$. Work inside~$\ideal{I}$. By \cref{prop:ATR:Laver:fund}, for every $S\le T$ that is $\Delta^1_1(X)$, and all~$n$, there is some $R\le_0 S$ that is $\Delta^1_1(X)$ such that $R\force x_G\in \mathcal{U}_n$. 

		We define a fusion sequence $(T_i)$ starting with $T_0 = T$. Let $i<\omega$ and suppose that~$T_i$ has been defined. Recall the trees $(T_i)^i_\tau$ for $\tau\in P_i(T_i)$ that partition $[T_i]$ into finitely many pieces (\cref{def:fusion_partition}). For each $\tau\in P_i(T_i)$, we find some $S^i_\tau\le_0 (T_i)^i_\tau$ such that $[S^i_\tau]\subseteq \+U_i$ (again, in the sense of $\ideal{I}$), and we let $T_{i+1} = \bigcup_{\tau\in P_i(T_i)} S^i_\tau$. Then $T_{i+1}\le_i T_i$ and $[T_{i+1}]\subseteq \+U_i$. 

		We can ensure that each $T_i$ is $\Delta^1_1(X)$; however, we need the entire sequence $(T_i)$ to exist. By \cite[Thm.\:VIII.4.11]{simpson}, $\atr$ together with $\Pi^1_1$-induction proves that for any $X$, the collection $\Delta^1_1(X)$ satisfies $\Sigma^1_1$-dependent choice, essentially because of $\Sigma^1_1(X)$-bounding: we can perform an ``effective'' construction in the sense of higher computability within $\Delta^1_1(X)$. Since $\ideal{I}$ is an $\omega$-model, it satisfies $\Pi^1_1$-induction. Hence, we can apply $\Sigma^1_1$-dependent choice in $\Delta^1_1(X)$ (again, in the sense of~$\ideal{I}$), and get the entire sequence $(T_i)$ to be $\Delta^1_1(X)$, in particular, $(T_i) \in \ideal{I}$. 

		We can therefore let $S = \bigcap_i T_i$; it is in~$\ideal{I}$. By \cref{lem:fusion}, $S$ is a Laver tree. 
		For all $i$, $S\le T_{i+1}$. We have $S\force_{\laver(\ideal{I})} x_G\in \+U_i$, so overall, $S\force_{\laver(\ideal{I})}x_G\in \+H$. 
	\end{proof}

	We recall the notion of a name (or a code) of a partial continuous function~$\Phi$ on Baire space: a set of pairs $(\sigma,\tau)$ such that $\Phi[\sigma]\subseteq [\tau]$, sufficient to determine the function. When such a name is computably enumerable, it is sometimes called a \emph{Turing functional}. As with the notions of $F_\sigma$ and $G_\delta$ sets, we say that a continuous function is $\ideal{I}$-computable if it has a name in~$\+I$. We will use notation from computability, and use~$\Phi$ to denote both the function and some name for it. For $\rho\in \omega^{\le \omega}$, we write $\Phi(\rho,m)=n$ when there is some pair $(\sigma,\tau)$ in the name such that $\rho\succeq \sigma$ and $\tau(m)=n$. We write $\Phi(\rho,m)\converge$ if $\Phi(\rho,m)=n$ for some~$n$. We let $\Phi(\sigma)$ be the longest string $\tau$ with $|\tau|\le |\sigma|$ and $\Phi(\sigma,m) = \tau(m)$ for all $m<|\tau|$. 	For $x\in \baire$, we say that $\Phi(x)$ is \emph{total} if $\Phi(x,m)\converge$ for all~$m$. The domain of~$\Phi$ --- the collection of~$x$ for which $\Phi(x)$ is total --- is $G_\delta$, indeed $\Pi^0_2$ relative to a name of~$\Phi$. For any Turing ideal~$\ideal{I}$ and any~$x$, the elements of $\ideal{I}[x]$ are precisely the reals $\Phi(x)$, where $\Phi$ is $\ideal{I}$-computable (has a name in~$\ideal{I}$) and $\Phi(x)$ is total. 

	Toward the Laver property for $\laver(\ideal{I})$ (\cref{def:Laver_property}) we use some notation first introduced in \cite{GM:dimension}. For a function $h\in \baire$, we let 
	\[
		h^\omega = \left\{ x\in \baire  \,:\,  x< h \right\}, 
	\]
	also sometimes denoted by $\prod_n h(n)$. If for some bound~$h$ we have $\Phi\colon \baire \to h^\omega$, then we will have $\Phi(x,m)<h(m)$ whenever it is defined. 

	\begin{lemma} \label{lem:atr:LaverProperty}
		If $\ideal{I}$ is a countable $\omega$-model of $\atr$, then $\laver(\ideal{I})$ has the Laver property over~$\ideal{I}$. 
	\end{lemma}
	
	\begin{proof}
		Let $h\in \ideal{I}$ be a bound, and let $\Phi\colon \baire\to h^\omega$ be $\ideal{I}$-computable. We need to show that if $G\subset \laver(\ideal{I})$ is sufficiently generic and $\Phi(x_G)$ is total then $\Phi(x_G)$ has a trace in~$\ideal{I}$. Let $T\in \laver(\ideal{I})$. By \cref{lem:atr:forcing_totality_or_divergence}, we may assume that~$T$ forces that $\Phi(x_G)$ is total. 

		\smallskip

		We define a fusion sequence $(T_i)$ starting with $T_0 = T$. Given~$T_i$, we again consider the partition $(T_i)^i_\tau$ for $\tau\in P_i(T_i)$. For each $m< h(i)$ let 
		\[
			V_m = \left\{ x\in \baire \,:\,  \Phi(x,i)=m \right\}. 
		\]
		Now by assumption, $[T]\cap \ideal{I}\subseteq \bigcup_{m<h(i)} V_m$. For each $\tau\in P_i(T_i)$, we find $S^i_\tau\le_0 (T_i)^i_\tau$ in~$\ideal{I}$ such that for some $m = m^i_\tau$ we have $[S^i_\tau]\cap \ideal{I}\subseteq V_m$. We let $T_{i+1} = \bigcup_{\tau\in P_i(T_i)} S^i_\tau$; we let $S = \bigcap_i T_i$; as in the previous proof, we can arrange so that $S\in \ideal{I}$. For each~$i$ we let
		\[
			F_i = \left\{ m^i_\tau  \,:\,  \tau\in P_i(T_i) \right\}. 
		\]
		Then $|F_i| \le |P_i(T_i)| = 2^i$ (see \cref{def:P_i_of_a_Laver_tree}) and $(F_i)\in \ideal{I}$ since the entire construction lives inside~$\ideal{I}$; and $S$ forces that $(F_i)$ is a trace of $\Phi(x_G)$. 
	\end{proof}
	
	Now \cref{prop:laver_adds_dominating_functions,cor:Laver_property_and_SME} imply:

	\begin{proposition} \label{prop:atr:separation_via_laver}
		If $\ideal{I}$ is a countable $\omega$-model of~$\atr$ and $G\subset \laver(\ideal{I})$ is sufficiently generic, then $x_G\in \domm{\ideal{I}}\setminus \SME{\ideal{I}}$. 	
	\end{proposition}

\subsection{The insufficiency of Laver forcing over $\hyp$}

Dan Turetsky showed that \cref{prop:atr:separation_via_laver} fails when $\ideal{I}= \hyp$  is the ideal of hyperarithmetic sets. The pertinent facts are:
\begin{itemize}
	\item If $x$ computes a list of reals containing all $\Delta^1_1$ reals, then $x\in \SNE{\hyp}$. 
	\item If Kleene's $\+O$ is $\Sigma^0_2(x)$, and $x \in \domm{\hyp}$, then $x$ computes a list containing all $\Delta^1_1$ reals. 
\end{itemize}  	

See \cite{listing} for details. 

\begin{proposition}[Turetsky]\label{prop:dans}
	For all sufficiently generic $G\subset \laver({\hyp})$, Kleene's~$\+O$ is $\Sigma^0_2(x_G)$. 
\end{proposition}
	
\begin{proof}
	Let $G\subset \laver({\hyp})$ be sufficiently generic. We show that for any $\Pi^0_1$ set $\+P\subseteq \baire$, $\+P$ contains a $\Delta^1_1$ real if and only if~$\+P$ contains a real dominated by~$x_G$. By the Spector-Gandy theorem, the collection of (indices of) $\Pi^0_1$ sets $\+P$ that have hyperarithmetic elements is $\Pi^1_1$-complete, and so is recursively isomorphic with Kleene's~$\+O$. On the other hand, for any $y\in \baire$, the collection of (indices of) $\Pi^0_1$ sets which have an element \emph{majorised} by~$y$ is $\Pi^0_1(y)$, as the space $y^\omega = \prod_n y(n)$ is effectively compact relative to~$y$; it follows that the collection of $\Pi^0_1$ sets which have an element dominated by~$x$ is $\Sigma^0_2(x)$. 

	One direction of the equivalence is implied by \cref{prop:laver_adds_dominating_functions}. In the other direction, suppose that $T\subseteq \omega^{<\omega}$ is a computable tree, and that $[T]$ contains no $\Delta^1_1$ element. Then $[T]$ contains no element dominated by a~$\Delta^1_1$ real (use relative effective compactness as above, and the fact that $\hyp$ is a jump ideal). Let $S\in \laver(\hyp)$. The leftmost path~$h$ of~$S$ is $S$-computable, so is~$\Delta^1_1$. Since $h^{\omega}$ is compact, and $h^\omega\cap [T] = \emptyset$, there is some~$n$ such $[T]$ contains no element~$x$ with $x<h\upto{n}$. Then the full subtree $R = S\upto(h\upto n)$ forces that $[T]$ contains no element majorised by~$x_G$. To improve this to domination, repeat the argument, but fixing some finite initial segment.
\end{proof}

\subsection{Separating dominating from strong meagre engulfing, the general case}

Turetsky's \cref{prop:dans} shows that to show that $\SME{\hyp}\ne \domm{\hyp}$, we cannot use closed sets that all contain $\Delta^1_1$ elements. As indicated in the introduction, we use the idea of Kumabe of adding sets of ``bad strings'', as long as the sets are small. 

\begin{notation}
	For the rest of the section, we fix a countable Turing ideal $\ideal{I} \subseteq \baire$ which is $\lhyp$-downward closed.
\end{notation}
	
\begin{definition}
	We let $\bushy = \bushy(\ideal{I})$ be the collection of pairs $\bm{p}=(\tree{p},\bad{p})$ where:
		\begin{itemize}
			\item $\tree{p}\in \laver(\ideal{I})$, and
			\item  $\bad{p} \subseteq \tree{p}$ is $\omega$-closed and upward closed in~$\tree{p}$, $\stem{\tree{p}} \notin\bad{p}$, and $\bad{p}$ is $\Pi^1_1(\ideal{I})$.\footnote{As with $\le_{\T(\ideal{I})}$, $\Pi^1_1(\ideal{I})$ means $\Pi^1_1(X)$ for some $X\in \ideal{I}$.}
		\end{itemize}
		A condition $\bm{q}$ extends a condition $\bm{p}$ if $\tree{q} \subseteq \tree{p}$ and $\bad{q} \supseteq \bad{p}\cap \tree{q}$.
	\end{definition}
	
	We say that $\bad{p}$ is the set of \emph{bad strings} (or the \emph{bad set}) of $\bm{p}$. We will shortly see that $\bm{p}$ forces the generic real outside of $\open{\bad{p}}$. 

\begin{notation}
		For $\bm{p}\in \bushy$ we let 
	\[
		\stm{p} = \stem{\tree{p}}. 
	\]
	Observe that if $\bm{q}$ extends~$\bm{p}$ then $\stm{q}\succeq \stm{p}$. 
\end{notation}

	\begin{lemma}\label{lem:containslaver}
		Let $\bm{p} \in \bushy$. Then $\tree{p} \setminus \bad{p}$ is a Laver tree with stem $\stm{p}$.
	\end{lemma}

	\begin{proof}
		The set $\tree{p} \setminus \bad{p}$ is a tree because $\tree{p}$ is a tree and $\bad{p}$ is upwards closed in~$\tree{p}$. To see that it is a Laver tree with stem $\stm{p}$, let $\sigma \succeq \stm{p}$ be in $\tree{p} \setminus \bad{p}$. Since $\tree{p}$ is a Laver tree with stem $\stm{p}$, $\sigma$ has infinitely many children in~$\tree{p}$. Since $\sigma\notin \bad{p}$ and $\bad{p}$ is $\omega$-closed, only finitely many of these children are in $\bad{p}$.
	\end{proof}
	
	For any $\bm{p}\in \bushy$ and any $\sigma\in \tree{p}\setminus \bad{p}$, 
	\[
		\bm{p}\upto{\sigma} = (\tree{p}\upto{\sigma}, \bad{p}\cap (\tree{p}\upto{\sigma}))
	\]
	is a condition in~$\bushy$ ($\bad{p}\cap (\tree{p}\upto{\sigma})$ is $\omega$-closed by \cref{rmk:closure_stays_inside_tree}(3)); it extends~$\bm{p}$. Together with \cref{lem:containslaver},  it follows that for a sufficiently generic $G\subset \bushy$, 
	\[
		x_G = \bigcup \left\{ \stm{p} \,:\,  \bm{p}\in G \right\}
	\]
	is an element of~$\baire$. 

	For $\bm{p}\in \bushy$, we let
	\[
		[\bm{p}] = [\tree{p}\setminus{\bad{p}}] = [\tree{p}]\setminus \open{\bad{p}}. 
	\]
	Observe that if $\bm{q}\le \bm{p}$ then $\tree{q}\setminus \bad{q}\subseteq \tree{p}\setminus \bad{p}$, so $[\bm{q}]\subseteq [\bm{p}]$. 
	
	\begin{lemma}\label{lem:avoidbad}
		For any $\bm{p} \in \bushy$, $\bm{p}\force_{\bushy} x_G\in [\bm{p}]$.
	\end{lemma}

	\begin{proof}
		Let $G\subset \bushy$ be sufficiently generic, and suppose that $\bm{p}\in G$. Let $\sigma \prec x_G$. There is some $\bm{q}\in G$ with $\sigma\preceq \stm{q}$. We may assume that~$\bm{q}$ extends~$\bm{p}$. Since $\stm{q}\in \tree{q}\setminus \bad{q}$, we have $\sigma\in \tree{q}\setminus \bad{q}$; since $\bm{q}\le \bm{p}$, $\sigma\in \tree{p}\setminus\bad{p}$. 
	\end{proof}

	We are thus justified in saying that $\bm{p}$ \emph{strongly forces} $\vphi(x_G)$ if $\vphi(x)$ holds for all $x\in [\bm{p}]$. 

	\begin{proposition} \label{prop:bushy_Laver:dominating}
		If $G\subset \bushy$ is a sufficiently generic, then $x_G$ dominates all functions in~$\ideal{I}$.
	\end{proposition}
	
	\begin{proof}
		Let $\bm{p}\in \bushy$ and let $h\in \ideal{I}$. Let~$\bm{q}$ be the following extension of~$\bm{p}$:
		\begin{itemize}
			\item $\tree{q}$ is the set of $ \sigma\in \tree{p}$ such that for all~$i$ with $|\stm{p}|\le i < |\sigma|$, $\sigma(i)\ge h(i)$; and
			\item $\bad{q} = \bad{p}\cap \tree{q}$. 
		\end{itemize}
		Then $\bm{q}$ strongly forces that $x_G$ dominates~$h$. 
	\end{proof}



	Suppose that $T$ is a Laver tree, and that $A\subseteq T$ is upwards closed in~$T$. Suppose that $\stem{T}\notin \cl(A)$. Then $\stem{T}$ is not in the upward closure of $\cl(A)$ in~$T$: if $\sigma\prec \stem{T}$ is in~$\cl(A)$ then $\sigma\in A$ (consider an $\omega$-bushy tree above~$\sigma$ with leaves in~$A$), and then $\stem{T}\in A$, so $\stem{T}\in \cl(A)$.

	\begin{lemma} \label{lem:Bushy_Laver:deciding_Pi_2}
		Let $\+H$ be an $\ideal{I}$-computable $G_\delta$ set. The collection of conditions that strongly decide $x_G\in \+H$ is dense in~$\bushy$. 
	\end{lemma}
	
	\begin{proof}
		The extra difficulty in comparison with the proof of \cref{lem:atr:forcing_totality_or_divergence} is that we don't have \cref{prop:ATR:Laver:fund}, that is, in the very first step. This is why we need to have the bad sets. Let $(A_n)\in \ideal{I}$ be a name of~$\+H$. We may assume that each~$A_n$ is upward closed in $\omega^{<\omega}$. 

		Let $\bm{p}\in \bushy$. We ask: is there some $\bm{q}\le \bm{p}$ and some~$n$ such that $A_n\cap \tree{q}$ is $\omega$-small (i.e., not $\omega$-big) above $\stm{q}$? That is, such that $\stm{q}\notin \cl(A_n\cap \tree{q})$? If so, we let $C$ be the upward closure of $\cl(A_n\cap \tree{q})$ in~$\tree{q}$; by \cref{rmk:closure_stays_inside_tree}, $C$ is $\omega$-closed. As we just observed, $\stm{q}\notin C$. By relativising \cref{prop:closure_of_Pi11_sets}, $\cl(A_n\cap \tree{q})$ is $\Pi^1_1(\ideal{I})$, and so~$C$ is $\Pi^1_1(\ideal{I})$.  By \cref{prop:closure_operator}, $\bad{q}\cup C$ is $\omega$-closed. Hence,
		 \[
		 \bm{r} = (\tree{q}, \bad{q}\cup C)
		 \]
		 is a condition in~$\bushy$ that extends~$\bm{q}$ (and so extends~$\bm{p}$) and strongly forces $x_G\notin \+H$. 

		\smallskip

		Suppose that the answer is ``no'': for all $\bm{q}\le \bm{p}$, for all~$n$, $\stm{q}\in \cl(A_n\cap \tree{q})$. We define a fusion sequence $(T_i)$ starting with $T_{0} = \tree{p}$. To do this, we perform an effective construction in $L_{\omega_1^X}[X]$ (the smallest admissible set containing~$X$), where $X\in \ideal{I}$ computes $\tree{p}$ and $(A_n)$, and is such that $\bad{p}$ is $\Pi^1_1(X)$. 

		Suppose that $T_i$ has been defined. For each $\tau\in P_i(T_i)$, we simultaneously enumerate the sets $\bad{p}$ and $\cl(A_i\cap (T_i)^i_\tau)$. By \cref{prop:closure_of_Pi11_sets}, at some $X$-computable stage, we see that either $\tau\in \bad{p}$, or that $\tau\in \cl(A_i\cap (T_i)^i_\tau)$. Note that if $\tau\notin \bad{p}$ then $((T_i)^i_\tau, \bad{p}\cap (T_i)^i_\tau)$ is a condition extending $\bm{p}$, with stem~$\tau$. 
		\begin{itemize}
			\item If we first see that $\tau\in \bad{p}$, we let $S^i_\tau = (T_i)^i_\tau$. 
			\item Otherwise, by \cref{prop:closure_of_Pi11_sets}, we obtain a well-founded tree~$R$, $\omega$-bushy above~$\tau$, with leaves in $A_i\cap (T_i)^i_\tau$. We let $S^i_\tau$ be the result of adding to~$R$ all predecessors of~$\tau$, and all extensions in $(T_i)^i_\tau$ of leaves of~$R$. 
		\end{itemize}
		In either case, $S^i_\tau$ is $\Delta^1_1(X)$ and is a Laver tree with $S^i_\tau \le_0 (T_i)^i_\tau$. We let $T_{i+1} = \bigcup_{\tau\in P_i(T_i)} S^i_\tau$. Since $P_i(T_i)$ is finite, there is an $X$-computable stage by which we discover all the trees $S^i_\tau$, and so the construction can proceed to the next step. 

		Finally, we let $S = \bigcap_i T_i$. Then $S\in \laver(\ideal{I})$. Let $\bm{q} = (S,\bad{p}\cap S)$. For any~$i$, for any ~$\tau \in P_i(S)$, either $\tau\in \bad{p}$, or $[S^i_\tau]\subseteq \open{A_i}$; so $[T_{i+1}\setminus \bad{p}]\subseteq \open{A_i}$. It follows that $\bm{q}$ strongly forces $x_G\in \+H$. 
	\end{proof}
	
\begin{proposition} \label{prop:Laver_bushy:Laver_property}
	$\bushy = \bushy(\ideal{I})$ has the Laver property over~$\ideal{I}$. 
\end{proposition}

\begin{proof}
	The proof follows the proof of \cref{lem:atr:LaverProperty}, with the same modification as in the proof of \cref{lem:Bushy_Laver:deciding_Pi_2}. We may assume that $\bm{p}$ strongly forces that $\Phi(x_G)$ is total, where $\Phi\colon \baire\to h^\omega$ is partial continuous with name in~$\ideal{I}$ (and $h\in \ideal{I}$). During the construction, for each~$i$, for each $\tau\in P_i(T_i)$, we search until we either see that $\tau\in \bad{p}$, or that for some $m<h(i)$, the set of strings~$\sigma$ in $(T_i)^i_\tau$ with $\Phi(\sigma,i)=m$ is $\omega$-big above~$\tau$, and define $S^i_\tau$ accordingly. 
\end{proof}

As before, we obtain:

\begin{corollary} \label{cor:separating_dom_and_SNE_in_general}
	For any countable Turing ideal~$\ideal{I}$ that is downward closed under $\lhyp$, $\SME{\ideal{I}}\subsetneq \domm{\ideal{I}}$. 
\end{corollary}

\section{Interlude: Hechler forcing.} \label{sec:Hechler}

In set theory, one builds a real in $\SME{V} \setminus \SNE{V}$ by forcing with two steps: first adding a Cohen real, then a \emph{Hechler} real. In \cref{sec:Hechler_and_Cohen} we will describe an effectivization of the 2-step iteration $\cohen \star \hechler$, and we will use it to show that, if $\ideal{I}$ is closed under relative hyperarithmeticity, then there is some real in $\SME{\ideal{I}} \setminus \SNE{\ideal{I}}$. Here we show an intermediate result which is helpful to understand the notion of forcing that we will introduce in \cref{sec:Hechler_and_Cohen}.

\begin{definition} \label{def:Hechler_forcing}
  	A \emph{Hechler condition} is a pair $(\tau,f)$ where $\tau\in \omega^{<\omega}$ and $f\in \baire$. A Hechler condition $(\tau,f)$ extends another condition $(\sigma, g)$ if $\tau\succeq \sigma$, $f\ge g$ (for all~$n$, $f(n)\ge g(n)$), and for every~$i$ with $|\sigma|\le i < |\tau|$ we have $\tau(i)\ge g(i)$. 
\end{definition}

For a Hechler condition $(\tau,f)$ we let $T_{\tau,f}$ be the collection of all $\sigma\in \omega^{<\omega}$ comparable with~$\tau$ such that for all~$i$ with $|\tau|\le i<|\sigma|$ we have $\sigma(i)\ge f(i)$. The closed set determined by a condition $(\tau,f)$ is $[T_{\tau,f}]$. 

\begin{remark}
	A variant of Hechler forcing deals with trees directly; a \emph{Hechler tree} is defined like a Laver tree, except that for every $\sigma$ extending the stem, for almost all~$n$, $\concat{\sigma}n\in T$. Each $T_{\tau,f}$ above is a Hechler tree, but not all Hechler trees are generated thusly, not even densely. Indeed, the two variants of Hechler forcing are in general not forcing equivalent. For a comparison of the two, see \cite{palumbo}. For the purpose of our version of the iteration $\cohen\star\hechler$, it is much easier to work with Hechler conditions rather than trees. 
\end{remark}

\begin{definition} \label{def:Hechler_with_bad_sets}
	Let $\ideal{I}$ be a Turing ideal. \emph{Hechler forcing with bad sets over $\ideal{I}$}, denoted by $\hechlerb{} = \hechlerb(\ideal{I})$, is the collection of triples $\bm{p} = (\stm{p},\fun{p}, \bad{p})$, where:
	\begin{itemize}
 		\item $(\stm{p}, \fun{p})$ is a Hechler condition, with $\fun{p}\in \ideal{I}$; 
 		\item $\bad{p} \subseteq \omega^{<\omega}$ is upwards closed in $\omega^{<\omega}$, $\bad{p}\in \ideal{I}$, and $\stm{p} \notin \cl(\bad{p})$.
 	\end{itemize} 	
 	A condition $\bm{q}$ extends a condition $\bm{p}$ if the Hechler condition $(\stm{q},\fun{q})$ extends $(\stm{p},\fun{p})$, and $\bad{p}\subseteq \bad{q}$. 
\end{definition}

\begin{remark}
	Note the difference with Laver forcing with bushy trees: here we are not assuming that~$\bad{p}$ is $\omega$-closed (for this reason, we can take $\bad{p}\in \ideal{I}$ rather than just being $\Pi^1_1(\ideal{I})$). The point is that because Hechler conditions are much less malleable than Laver trees (we cannot perform fusion), we will need to consider whether certain sets are $\omega$-big above various strings extending the stem, not only above the stem itself. This implies that we cannot just pass to the upwards closure of an $\omega$-closed set as we did above. 
\end{remark}

For basic facts about Hechler forcing with bad sets, we need to identify which finite extensions of the stem are ``permitted'' by a condition.

\begin{definition} \label{def:Hechler_with_bad_sets:E_p}
 For $\bm{p}\in \hechlerb$ we let 
\[
	E_{\bm{p}} = \left\{ \sigma\succeq \stm{p}  \,:\,  \sigma\in T_{\stm{p},\fun{p}}\setminus \cl(\bad{p}) \right\}. 
\]	
\end{definition}

\begin{lemma} \label{lem:Hechler_with_bad_sets:possible_finite_extensions}
	Let $\bm{p}\in \hechlerb$. The following are equivalent for $\sigma\in \omega^{<\omega}$: 
	\begin{enumerate}
		\item $\bm{p}\upto \sigma = (\sigma,\fun{p},\bad{p})\in \hechlerb$ and extends~$\bm{p}$; 
		\item There is some $\bm{q}\le \bm{p}$ in $\hechlerb$ such that $\sigma = \stm{q}$; 
		\item $\sigma\in E_{\bm{p}}$. 
	\end{enumerate}
\end{lemma}

\begin{proof}
	(2)$\Rightarrow$(3): suppose that $\bm{q}\le \bm{p}$. Since the Hechler condition $(\stm{q},\fun{q})$ extends the Hechler condition $(\stm{p},\fun{p})$, we have $\stm{q}\in T_{\stm{p},\fun{p}}$ (and $\stm{q}\succeq \stm{p}$). We have $\stm{q}\notin \cl(\bad{q})$ and $\bad{p}\subseteq \bad{q}$, so $\cl(\bad{p})\subseteq \cl(\bad{q})$ and $\stm{q}\notin \cl(\bad{p})$. 

	(3)$\Rightarrow$(1) is by definition of~$\hechlerb$. 
\end{proof}

We also observe:

\begin{lemma} \label{lem:Hechler_with_bad:existence_of_long_extensions}
	For all $m<\omega$, the collection of $\bm{p}\in \hechlerb$ with $|\stm{p}|\ge m$ is dense in~$\hechlerb$. 
\end{lemma}

\begin{proof}
	Let $\bm{p}\in \hechlerb$, and let $m > |\stm{p}|$. The collection of strings $\sigma\in T_{\stm{p},\fun{p}}$ of length $\ge m$ is $\omega$-big above~$\stm{p}$, whereas $\cl(\bad{p})$ is small above~$\stm{p}$, so there is some $\sigma\in E_{\bm{p}}$ of length $\ge m$. 
\end{proof}

For a filter $G\subset \hechlerb$ we let
\[
	x_G = \bigcup \left\{ \stm{p} \,:\,  \bm{p}\in G \right\};
\]
by \cref{lem:Hechler_with_bad:existence_of_long_extensions}, if~$G$ is sufficiently generic then $x_G\in \baire$. 

For $\bm{p}\in \hechlerb$ we let 
\[
	[\bm{p}] = [T_{\stm{p},\fun{p}}\setminus \bad{p}] = [T_{\stm{p},\fun{p}}]\setminus \open{\bad{p}}. 
\]

The argument of \cref{lem:avoidbad} gives:

\begin{lemma} \label{lem:Hechler_with_bad:strong_forcing_implies_forcing}
	Every $\bm{p}\in \hechlerb$ forces that $x_G\in [\bm{p}]$.
\end{lemma}

Again, we say that $\bm{p}$ strongly forces $\varphi$ if $\varphi(x)$ holds for every $x \in [\bm{p}]$.

\begin{proposition} \label{prop:Hechler_with_bad_sets:dominating}
 For any countable Turing ideal~$\ideal{I}$, if $G\subset \hechlerb(\ideal{I})$ is sufficiently generic then~$x_G$ dominates all functions in~$\ideal{I}$. 
\end{proposition}

\begin{proof}
	Let $\bm{p}\in \hechlerb$, and let $h\in \ideal{I}$. Let $g(m) = \max \{h(m),\fun{p}(m)\}$. Then $(\stm{p}, g,\bad{p})$ is in~$\hechlerb$, extends~$\bm{p}$, and strongly forces that $x_G$ dominates~$h$. 
\end{proof}

\subsection{Strong null engulfment and computing random reals}

Recall that above we showed that a real that we constructed was not strongly meagre engulfing because it did not compute any $\ideal{I}$-Cohen real. The exact same argument holds for the null ideal. For a Turing ideal~$\ideal{I}$, a real $x$ is \emph{$\ideal{I}$-random} if it is not an element of any $\ideal{I}$-computable null (a null set with a name in~$\ideal{I}$). As with Cohen reals, for a jump ideal, all arithmetical notions of randomness coincide; for a general ideal, this notion of randomness is sometimes called $\ideal{I}$-\emph{Schnorr} randomness. 

Just as with the meagre ideal, there is an effective morphism from the Weihrauch problem $\spill{\ideal{N}}$ to the problem $\capture{\ideal{N}}$ (see \cite[Proposition 3.9]{GKT} for details), which shows that for any ideal~$\ideal{I}$, if $x\in \SNE{\ideal{I}}$ then there is some $\ideal{I}$-random real that is $\ideal{I}$-computable from~$x$.  In the next section, when we construct a real in $\SME{\ideal{I}}\setminus \SNE{\ideal{I}}$, we will do so by constructing a real that does not $\ideal{I}$-compute any $\ideal{I}$-random real. 

In this section we show the weaker result, that Hechler forcing with bad sets adds a dominating real that does not $\ideal{I}$-compute a random. To get the real into the class $\SME{\ideal{I}}$ we will need to work harder (as mentioned, force with a 2-step iteration of Cohen and Hechler), so the argument in this section serves as a warm-up to the next section.

\medskip

\subsubsection*{$\sigma$-centered notions of forcing} 
The reason that (in set theory) a Hechler real does not add randoms is that Hechler forcing is \emph{$\sigma$-centered}. That is, Hechler forcing can be partitioned into countably many sets, such that any two conditions in the same part of the partition are compatible. Namely, any two Hechler conditions $(\tau,f)$ and $(\tau,f')$ with the same ``stem'' are compatible, since the condition $(\tau,\max\{f,f'\})$ extends them both. 

The ``bad set'' variant of Hechler forcing is $\sigma$-centered as well:
	\begin{lemma}
	If $\bm{p},\bm{q}\in \hechlerb$ and $\stm{p}= \stm{q}$ then $\bm{p}$ and $\bm{q}$ are compatible in $\hechlerb$. 
\end{lemma}

\begin{proof}
	Define a condition $\bm{r}$ by letting $\stm{r} = \stm{p}$,  $\fun{r}=\max\{\fun{p}, \fun{q}\}$ (the pointwise maximum of the two functions), and $\bad{r} = \bad{p}\cup \bad{q}$. By \cref{prop:closure_operator}, $\bad{r}$ is a condition, and it extends both $\bm{p}$ and~$\bm{q}$.
\end{proof}

\subsection{Forcing totality and avoiding randoms}

For the proof of the following lemma, we use the fact that if $(\tau,f)$ is a Hechler condition, then $T_{\tau,f}$ is a Hechler tree, so for any $A\subseteq \omega^{<\omega}$ and any $\sigma\succeq \tau$ in $T_{\tau,f}$, $A$ is $\omega$-big above~$\sigma$ if and only if $A\cap T_{\tau,f}$ is $\omega$-big above~$\sigma$ (this is why in the definition of $\hechlerb$ we do not require $\bad{p}\subseteq T_{\stm{p},\fun{p}}$). 

\begin{lemma} \label{lem:Hechler:forcing_open}
	Let $A\in \ideal{I}$ be a name of an open set, and let $\bm{p}\in \hechlerb$. 
	\begin{enumerate}
		\item If $E_{\bm{p}}\subseteq \cl(A)$ then $\bm{p}\force x_G\in \open{A}$. 
		\item If $E_{\bm{p}}\nsubseteq \cl(A)$ then~$\bm{p}$ has an extension which strongly forces that $x_G\notin \open{A}$. 
	\end{enumerate}
\end{lemma}

\begin{proof}
	For~(1), let $\bm{q}\le \bm{p}$. Since $\stm{q}\in E_{\bm{p}}$, by assumption, $A$ is $\omega$-big above~$\stm{q}$. Since $\bad{q}$ is not $\omega$-big above~$\stm{q}$, we can find some $\sigma\in T_{\stm{q},\fun{q}}$ which is in $A\setminus \cl(\bad{q})$. So $(\sigma,\fun{q},\bad{q})$ is a condition extending~$\bm{q}$ which strongly forces that $x_G\in \open{A}$, and the collection of conditions that strongly force $x_G\in \open{A}$ is dense below~$\bm{p}$; it follows that $\bm{p}\force x_G\in \open{A}$. 

	For~(2), suppose that $\sigma\in E_{\bm{p}}\setminus \cl(A)$. Then $(\sigma,\fun{p},\bad{p}\cup A)$ extends~$\bm{p}$ and strongly forces that $x_G\notin \open{A}$. 
\end{proof}

If~$T$ is a Laver tree and $A\cap T$ is $\omega$-big above $\stem{T}$, then we have some $S\le_0 T$ that strongly forces that $x_G\in \open{A}$. In contrast, \cref{lem:Hechler:forcing_open} gives forcing into~$\open{A}$ rather than strong forcing, because we cannot throw infinitely many children out. In terms of bigness, the Laver forcing argument works because a set is either $\omega$-big or $\omega$-small. We would be able to obtain a similar result for Hechler forcing if the complement of an $\omega$-big set was always $\omega$-small, but this is readily seen to be false. This and related matters are explored in \cite{laverpartition}.

\begin{proposition} \label{prop:Hechler:no_randoms}
	Suppose that $\ideal{I}$ is closed under $\lhyp$. Then for sufficiently generic $G\subset \hechlerb(\ideal{I})$, $\ideal{I}[x_G]$ contains no $\ideal{I}$-random reals. 
\end{proposition}

\begin{proof}
	Let $\Phi$ be an $\ideal{I}$-partial computable function from~$\baire$ to~$\cantor$, and let $\bm{p}\in \hechlerb$. For each~$n$, let $A_n$ be the collection of $\sigma\in\omega^{<\omega} $ such that $|\Phi(\tau)|\ge n$. If there is some~$n$ such that $E_{\bm{p}}\nsubseteq \cl(A_n)$, then by \cref{lem:Hechler:forcing_open}, $\bm{p}$ has an extension that forces that $\Phi(x_G)$ is not total. Assume otherwise; we show that $\bm{p}$ forces that $\Phi(x_G)$ is not $\ideal{I}$-random. 

	 Let $X\in \ideal{I}$ that computes $\fun{p}$, $\bad{p}$, and~$\Phi$. As in the proof of \cref{lem:Bushy_Laver:deciding_Pi_2}, we perform an effective construction in the least admissible set containing~$X$. For each $\nu\in 2^{<\omega}$ we let 
	\[
		V_\nu = \left\{ \sigma  \,:\,  \nu\preceq \Phi(\sigma) \right\}. 
	\]
	These sets are $X$-computable, uniformly, so the sets $\cl(V_\nu)$ are $\Pi^1_1(X)$, uniformly; and $\cl(\bad{p})$ is $\Pi^1_1(X)$. For $\sigma\in T_{\stm{p}, \fun{p}}$ extending $\stm{p}$ we define $P_\sigma\subseteq 2^{<\omega}$ as follows: 
	\begin{itemize}
		\item If $\sigma\notin \cl(\bad{p})$, then 
		\[
			P_\sigma = \left\{ \nu  \,:\,  \sigma\in \cl(V_\nu) \right\}. 
		\]
		\item If $\sigma\in \cl(\bad{p})$ then $P_\sigma = 2^{<\omega}$. 
	\end{itemize}
	Since $T_{\stm{p},\fun{p}}$ is $X$-computable, the sets $P_\sigma$ are $\Pi^1_1(X)$, uniformly (as long as we don't see $\sigma$ enter $\cl(\bad{p})$, we enumerate~$\nu$ into~$P_\sigma$ when we see~$\sigma$ enter $\cl(V_\nu)$; once we see~$\sigma$ enter $\cl(\bad{p})$, we throw all the strings into~$P_\sigma$). The main point is that each $P_\sigma$ is infinite: this is clear when $\sigma\in \cl(\bad{p})$. If $\sigma\notin \cl(\bad{p})$, i.e., if $\sigma\in E_{\bm{p}}$, then by assumption, for all~$n$, $A_n$ is $\omega$-big above~$\sigma$. For each~$n$, since there are only finitely many binary strings of length~$n$, by \cref{prop:closure_operator}, there is some $\nu\in 2^{<\omega}$ of length~$n$ such that $V_\nu$ is $\omega$-big above~$\sigma$, i.e., such that $\nu\in P_\sigma$. 

	By admissibility, there is some $X$-computable stage by which we see that every~$P_\sigma$ is infinite. This gives us uniformly $\Delta^1_1(X)$, infinite sets $R_\sigma\subseteq P_\sigma$. By assumption on~$\ideal{I}$, the array $(R_\sigma)$ is in~$\ideal{I}$. With these we can find a set $S\subset 2^{<\omega}$ in~$\ideal{I}$ with finite weight ($\sum_{\sigma\in S} 2^{-|\sigma|}$ is finite) such that for all~$\sigma$, $S\cap R_\sigma$ (and so $S\cap P_\sigma$) is infinite.

	The set~$S$ can be considered as a Solovay test; so any $y\in \cantor$ with infinitely many initial segments in~$S$ is not $\ideal{I}$-random. We claim that $\bm{p}$ forces that $\Phi(x_G)$ has infinitely many initial segments in~$S$. 

	To see this, let $\bm{q}\le \bm{p}$ and let $n<\omega$. Let $\nu\in S\cap P_{\stm{q}}$ have length $\ge n$. Since $\stm{q}\notin \cl(\bad{p})$, $V_{\nu}$ is $\omega$-big above~$\stm{q}$. So there is some $\sigma\in V_{\nu}\cap E_{\bm{q}}$, so $\bm{q}$ has an extension~$\bm{r}$ with $\stm{r}\in V_\nu$, i.e., $\Phi(\stm{r})\succeq \nu$; so~$\bm{r}$ (strongly) forces that $\nu\prec \Phi(x_G)$.
\end{proof}

\begin{remark} \label{rmk:Hechler_in_atr}
	As with Laver forcing, if $\ideal{I}$ is an $\omega$-model of $\atr$, then we can dispose of the bad sets, and force with Hechler trees. To see why, we observe that the argument of \cref{prop:ATR:Laver:fund} shows that if $B\in \ideal{I}$ is a name of an open set, then for all~$\sigma$, either $B$ is $\omega$-big above~$\sigma$, or there is a Hechler tree $T\in \ideal{I}$ with stem~$\sigma$ such that $T\cap B = \emptyset$. 	  
\end{remark}

\section{Iterating Cohen and Hechler forcing} \label{sec:Hechler_and_Cohen}
	
We turn to the proof of the second part of \cref{thm:main}: if $\ideal{I}$ is closed under $\lhyp$, then $\SNE{\ideal{I}}\subsetneq \SME{\ideal{I}}$. As discussed in the previous section, to show that the real~$z$ that we construct is not in $\SNE{\ideal{I}}$, we will ensure that $\ideal{I}[z]$ contains no $\ideal{I}$-random. The following lemma explains why we will use a 2-step iteration, adding a Cohen real~$x$ and a Hechler real over $\ideal{I}[x]$: 

\begin{lemma}\label{lem:truss}
	Let $\ideal{I}$ be any Turing ideal. If $x$ is $\ideal{I}$-Cohen and $y$ is $\ideal{I}[x]$-dominating, then $y \in \SME{\ideal{I}}$.
\end{lemma}
	
Lemma \ref{lem:truss} follows from \cite[Proposition 4.9]{GKT}. The latter Proposition is an effective-morphism-based proof of the equalities in Cichon's diagram summarized in \cite[Corollary 4.10]{GKT}, originally due to Truss, A. Miller, and Fremlin. The morphism argument relies on the sequential composition of Weihrauch problems. 

\subsection{Bigness for sets of pairs}

In set theory, the 2-step iteration theorem states that if $\mathbb{P}$ is a notion of forcing, and $\mathbb{Q}$ is a $\mathbb{P}$-name for a notion of forcing, then there is a single notion of forcing $\mathbb{P}\star \mathbb{Q}$ (in the ground model) which is forcing equivalent to forcing with~$\mathbb{P}$ and then with the interpretation of $\mathbb{Q}$ in $V^{\mathbb{P}}$. This equivalence is not, in general, effective, and in computability theory, there are situations in which it is important that the two steps are performed sequentially (see for example \cite{Introreducible}). For our purposes, though, it would be simpler to combine the steps into a single notion of forcing. Thus, conditions in our notion of forcing will consist of a Cohen condition, and a continuous Cohen-name for a Hechler condition in the extension. To be able to decide open sets, we will again need bad sets, in this case sets of pairs of strings (a Cohen condition, and a possible stem of a Hechler condition). The main work therefore is in adapting the notion of $\omega$-bigness and closure to this setting. For a similar situation regarding iterations of bushy tree forcing see~\cite{CGM}.

\begin{notation}
	For $(\sigma,\tau), (\sigma',\tau')\in \upairs$, we write:
	\begin{itemize}
		\item $(\sigma,\tau)\preceq (\sigma',\tau')$ if $\sigma\preceq \sigma'$ and $\tau\preceq\tau'$. 
		\item $(\sigma,\tau)\prec (\sigma',\tau')$ if $\sigma\prec \sigma'$ and $\tau\prec\tau'$ (so strict extension in both coordinates). 
	\end{itemize}
	We say that $B\subseteq \upairs$ is \emph{upwards closed} if for all $(\sigma,\tau)\preceq (\sigma',\tau')$ in $\upairs$, if $(\sigma,\tau)\in B$ then $(\sigma',\tau')\in B$. 
\end{notation}

\begin{definition}
	Recall that a set $A\subseteq 2^{<\omega}$ is \emph{dense above} a string $\sigma\in 2^{<\omega}$ if every $\sigma'\succeq \sigma$ has an extension in~$A$. 

	We say that $A\subseteq 2^{<\omega}$ is \emph{somewhere dense above $\sigma$} if it is dense above some $\sigma'\succeq \sigma$. Otherwise we say that it is \emph{nowhere dense} above~$\sigma$.   
\end{definition}

\begin{definition} \label{def:omega-closed_for_pairs}
	A set $B\subseteq 2^{<\omega}\times \omega^{<\omega}$ is \emph{$\omega$-closed} if for every pair $(\sigma,\tau)\in \upairs$, if
	\[
		\left\{ \sigma'\succeq\sigma  \,:\,  (\exists^\infty n)\,\,(\sigma', \concat{\tau}{n})\in B \right\}
	\]	
	is somewhere dense above~$\sigma$, then $(\sigma,\tau)\in B$. 
\end{definition}

We describe the closure of a set of pairs, both via ranks and via ``bigness witnesses'' analogous to $\omega$-bushy trees. We start with the former. 

\begin{definition} \label{def:rank_for_set_of_pairs}
	Let $B\subseteq 2^{<\omega}\times \omega^{<\omega}$. By induction on ordinals~$\alpha$ we define the relations $\rk_B(\sigma,\tau) \le \alpha$ for $(\sigma,\tau)\in\upairs$: 
	\begin{itemize}
		\item $\rk_B(\sigma,\tau)\le 0$ if $(\sigma,\tau) \in B$. 

		\item For $\alpha>0$, $\rk_B(\sigma,\tau)\le \alpha$ if the set 
		\[
			\left\{ \sigma'\succeq \sigma  \,:\,  (\exists^\infty n)\,\,\rk_B(\sigma',\concat{\tau}n)<\alpha \right\}
		\]
		is somewhere dense above~$\sigma$ (where again, $\rk_B(\sigma,\tau)<\alpha$ means $\rk_B(\sigma,\tau)\le \beta$ for some $\beta<\alpha$). 
	\end{itemize}
	As before, we write $\rk_B(\sigma,\tau)<\infty$ if $\rk_B(\sigma,\tau)\le \alpha$ for some~$\alpha$, and we let $\rk_B(\sigma,\tau)$ denote the least such~$\alpha$; if there is no such~$\alpha$ we write $\rk_B(\sigma,\tau)=\infty$. 
\end{definition}

We turn to bigness witnesses. To motivate the definition, recall that for sets of strings, the ``canonical'' bigness witness was the collection of strings with strictly decreasing ranks. 

\begin{definition} \label{def:bigness_witness_for_sets_of_pairs}
	Let $(\rho,\nu)\in \upairs$. A \emph{bigness witness above $(\rho,\nu)$} is a tree~$R$ of nonempty finite sequences whose entries are pairs in $\upairs$, satisfying: 
	\begin{itemize}
		\item The only sequence of length~1 on~$R$ is $\smallseq{(\rho,\nu)}$ (we call that sequence the root of~$R$);

		\item If $(\sigma,\tau)$ and $(\sigma',\tau')$ are successive entries of a sequence $s\in R$, then $(\sigma,\tau)\prec (\sigma',\tau')$, and $|\tau'| = |\tau|+1$; 

		\item	Suppose that $s\in R$ is not a leaf of~$R$; let $(\sigma,\tau)$ be the last entry of~$s$. Then 
		\[
			\left\{ \sigma'\succ \sigma  \,:\, (\exists^\infty n)\,\,\concat{s}{\smallseq{(\sigma',\concat{\tau}n)}} \in R \right\}
		\]
		is somewhere dense above~$\sigma$. 


	\end{itemize}
\end{definition}

\begin{lemma} \label{lem:equivalence_of_rank_and_bigness_witnesses_for_pairs}
	Let $B\subseteq \upairs$, and let $(\sigma,\tau)\in \upairs$. The following are equivalent:
	\begin{enumerate}
		\item $\rk_B(\sigma,\tau)<\infty$; 
		\item There is a well-founded bigness witness~$R$ above $(\sigma,\tau)$, such that for every leaf~$s$ of~$R$, the last entry of~$s$ is in~$B$. 
	\end{enumerate}
\end{lemma}

When these equivalent conditions hold, we say that \emph{$B$ is $\omega$-big above $(\sigma,\tau)$}. 

\begin{proof}
This is similar to the proof of \cref{lem:equivalence_of_ranked_and_bigness}. In one direction, suppose that $R$ is a well-founded bigness witness above $(\sigma,\tau)$, where the last entry of any leaf of~$R$ is in~$B$. By induction on $\rk_R(s)$ for $s\in R$, we see if $(\rho,\nu)$ is the last entry of~$s$, then $\rk_B(\rho,\nu)\le \rk_R(s)$. 

In the other direction, suppose that $\rk_B(\sigma,\tau)< \infty$. We define a bigness witness $R$ above $(\sigma,\tau)$ recursively: we start with the root $\seq{(\sigma,\tau)}$. Suppose that $s\in R$, and let $(\rho,\nu)$ be the last entry of~$s$. If $(\rho,\nu)\in B$ then~$s$ is a leaf of~$R$. Otherwise, we let the children of~$s$ on~$R$ be all the sequences $\concat{s}{\smallseq{(\rho',\concat{\nu}{n}})}$ where $\rho'\succ \rho$ and $\rk_B(\rho',\concat{\nu}{n})< \rk_B(\rho,\nu)$. 
\end{proof}

\begin{definition} \label{def:omega_closure_of_sets_of_pairs}
	For $B\subseteq \upairs$, we let $\cl(B)$ be the collection of all $(\sigma,\tau)$ such that $B$ is $\omega$-big above $(\sigma,\tau)$. 
\end{definition}

The analogue of \cref{lem:characterisation_of_omega-closed} holds, with a similar argument:

\begin{lemma} \label{lem:characterisation_of_omega_closed_sets_of_pairs}
	A set $B\subseteq \upairs$ is $\omega$-closed if and only if $B = \cl(B)$. 
\end{lemma}

We also obtain the two ``bigness properties'':

\begin{lemma} \label{lem:concatenation_and_preservation_of_smallness_for_pairs}
	Let $A,B,C\subseteq \upairs$.
	\begin{enumerate}
		\item Suppose that $B$ is $\omega$-big above~$(\sigma,\tau)$, and that for all $(\rho,\nu)\in B$, $C$ is $\omega$-big above $(\rho,\nu)$.  Then~$C$ is $\omega$-big above $(\sigma,\tau)$. 

		\item If $A\cup B$ is $\omega$-big above $(\sigma,\tau)$, then either~$A$ or~$B$ are $\omega$-big above $(\sigma,\tau)$. 
	\end{enumerate}
\end{lemma}

\begin{proof}
	(1) follows from the same concatenation argument of \cref{lem:concatenation_and_preservation_of_smallness}, using bigness witnesses. 

	For~(2), we again show by induction on $\rk_{A\cup B}(\sigma,\tau)$ that
	\[
		\rk_{A\cup B}(\sigma,\tau) = \min \{ \rk_{A}(\sigma,\tau), \rk_{B}(\sigma,\tau) \}. 
	\]
	This certainly holds if $\rk_{A\cup B}(\sigma,\tau)=0$. Suppose that this holds for all pairs with $A\cup B$-rank below~$\alpha$; suppose that $\rk_{A\cup B}(\sigma,\tau) = \alpha$. For $X\in \{A,B,A\cup B\}$  let 
	\[
		D_X = \left\{ \sigma'\succ \sigma  \,:\, (\exists^\infty n)\,\,\rk_X(\sigma', \concat{\tau}n) <\alpha   \right\}. 
	\]
	By induction, $D_{A\cup B} = D_A\cup D_B$. The union of two sets that are nowhere dense above~$\sigma$ is nowhere dense above~$\sigma$; since $D_{A\cup B}$ is somewhere dense above~$\sigma$, at least one of $D_A$ or~$D_B$ is somewhere dense above~$\sigma$, so $\rk_A(\sigma,\tau)\le \alpha$  or $\rk_B(\sigma,\tau)\le \alpha$. 
\end{proof}

And translated to the closure operator:

\begin{proposition} \label{prop:closure_operator_for_pairs}
	The operator~$\cl$ on subsets of~$\upairs$ satisfies:
	\begin{enumerate}
		\item $B\subseteq \cl(B)$; 
		\item $\cl(\cl(B)) = \cl(B)$; and
		\item $\cl(A\cup B) = \cl(A)\cup \cl(B)$. 
	\end{enumerate}
\end{proposition}

The arguments for \cref{prop:closure_of_Pi11_sets} give:

\begin{proposition} \label{prop:closure_of_Pi11_sets_of_pairs}
	Suppose that $B\subseteq \upairs$ is $\Pi^1_1$. 
	\begin{enumerate}
		\item For all $(\sigma,\tau)$, if $\rk_B(\sigma,\tau)<\infty$ then $\rk_B(\sigma,\tau)< \ock$.
		\item If $B$ is $\omega$-big above~$(\sigma,\tau)$, then there is a $\Delta^1_1$ bigness witness~$R$ above~$(\sigma,\tau)$ witnessing this. 
		\item $\cl(B)$ is $\Pi^1_1$. 
	\end{enumerate}
\end{proposition}




	
\subsection{The notion of forcing~$\iter$}

	We again fix a countable Turing ideal $\ideal{I}$ which is closed under $\lhyp$. We are ready to present our effective version of $\cohen \star \hechler$ over~$\ideal{I}$, with bad sets. We continue to use standard notation from computability regarding partial continuous functions and their names (codes, or ``functionals''). In particular, we will be using names for partial continuous functions from~$\cantor$ to~$\baire$. If~$f$ is such a name, then for each $\sigma\in 2^{<\omega}$ we let $f(\sigma)$ be the longest string $\tau\in \omega^{<\omega}$ such that $|\tau|\le |\sigma|$, and for all $m<|\tau|$, $f(\sigma,m)\converge = \tau(m)$. For such a name~$f$, and a string $\sigma\in 2^{<\omega}$, telling whether $\dom f$ (considered as a partial function on~$\cantor$) is comeagre in~$[\sigma]$ is arithmetical in~$f$: for all~$n$, $\left\{ \sigma'\succeq \sigma  \,:\,  |f(\sigma')|\ge n \right\}$ is dense above~$\sigma$ (in other words, as a Cohen condition, $\sigma$ forces that $f(x_G)$ is total). 

\begin{notation}
	Above, we only used the notation $\tau\le \rho$ for sequences of the same length. In this section we extend it as follows: we write $\tau\le \rho$ when $|\tau|\ge |\rho|$ and $\tau(m)\le \rho(m)$ for all $m<|\rho|$. Note that this is indeed a partial ordering on~$\omega^{<\omega}$.
\end{notation}

	\begin{definition}
		We define a forcing notion $\iter = \iter(\ideal{I})$ as follows.
		
		Elements of $\iter$ are quadruples $\bm{p}=(\mindom{p},\stm{p},\fun{p},\bad{p})$ such that
		\begin{enumerate}
			\item $\mindom{p}\in 2^{<\omega}$, and $\stm{p}\in \omega^{<\omega}$; 
			\item $\fun{p}\in \ideal{I}$ is a name of a partial continuous function from~$\cantor$ to~$\baire$ satisfying:
			\begin{itemize}
			 	\item $\dom \fun{p}$ is comeagre in~$[\mindom{p}]$; and
			 	\item $\fun{p}(\mindom{p}) \le  \stm{p}$. 
			 \end{itemize} 
			\item $\bad{p}\subseteq 2^{<\omega}\times \omega^{<\omega}$ is upwards closed, $\bad{p}\in \ideal{I}$, and $(\mindom{p}, \stm{p}) \notin \cl(\bad{p})$.  
		\end{enumerate}
		
		A condition $\bm{q}$ extends another condition $\bm{p}$ if:
		\begin{enumerate}
			\item $\mindom{q} \succeq \mindom{p}$ and $\stm{q}\succeq \stm{p}$; 
			\item $\bad{q} \supseteq \bad{p}$; and
			\item For all $\sigma\succeq \mindom{p}$,  $\fun{q}(\sigma) \geq \fun{p}(\sigma)$.
		\end{enumerate}
	\end{definition}
		
	Intuitively, one can view a condition $\bm{p}$ as consisting of the following components: $\mindom{p}$ is ``the Cohen part'', $\stm{p}$ is the finite part of a Hechler condition, $\fun{p}$ is a ``Cohen name'' for a function which will give us the infinitary part of a Hechler condition, and $\bad{p}$ is a Cohen name for a set of bad strings in $\omega^{<\omega}$ that the Hechler condition needs to avoid, as in the previous section. 

	\smallskip

	We now provide the basic properties of $\iter$. First, we observe that~$\iter$ is $\sigma$-centered (essentially because it is the iteration of two $\sigma$-centered notions of forcing).
	\begin{lemma}
		If $\bm{p}, \bm{q} \in \iter$ and $(\mindom{p}, \stm{p})=(\mindom{q}, \stm{q})$, then $\bm{p}$ and $\bm{q}$ are compatible in~$\iter$. 
	\end{lemma}
	\begin{proof}
		Suppose that $(\mindom{p}, \stm{p})=(\mindom{q}, \stm{q})$. By \cref{prop:closure_operator_for_pairs}, $(\mindom{p}, \stm{p})\notin \cl(\bad{p} \cup \bad{q})$. Now let $f$ be defined as 
		 \[
		 f(\sigma,i)=\max\{\fun{p}(\sigma,i), \fun{q}(\sigma,i)\}
		 \]
		 for every $\sigma$ and every $i < \min\{|\fun{p}(\sigma)|, |\fun{q}(\sigma)|\}$ (so $|f(\sigma)|=\min\{|\fun{p}(\sigma)|, |\fun{q}(\sigma)|\}$). Then $\bm{s}=(\mindom{p},\stm{p}, f,  \bad{p} \cup \bad{q})$ is a condition extending both $\bm{p}$ and $\bm{q}$.
	\end{proof}
	
	Our next goal is to identify the closed set determined by a condition. This is a bit more tricky, because $\upairs$ is not a tree. The following is an analogue of $T_{\tau,f}$ from the previous section. 

	\begin{definition} \label{def:C_p}
		For a condition $\bm{p}\in \iter$, we let $C_{\bm{p}}$ be the set of pairs $(\sigma,\tau)\in \upairs $ satisfying:
		\begin{itemize}
			\item $(\sigma, \tau)\succeq (\mindom{p}, \stm{p})$; and
			\item  $\tau\ge \fun{p}(\sigma)$. 
		\end{itemize}
		We let 
		\[
			E_{\bm{p}} = C_{\bm{p}}\setminus \cl(\bad{p}). 
		\]
	\end{definition}
	
	Observe that $C_{\bm{p}}$ is computable from~$\bm{p}$, so $C_{\bm{p}}\in \ideal{I}$. Observe also that $(\mindom{p},\stm{p})\in C_{\bm{p}}$. However, $C_{\bm{p}}$ is not closed downwards in $(\upairs,\preceq)$. Rather, if $(\sigma,\tau)\in C_{\bm{p}}$ then $(\sigma',\tau')\in C_{\bm{p}}$ whenever $\sigma'\succeq \sigma$ and $\stm{p}\preceq \tau'\preceq \tau$. 

	\begin{lemma} \label{lem:C_p_is_quite_big}
		For any $\bm{p}\in \iter$, $(\upairs) \setminus C_{\bm{p}}$ is $\omega$-small above any $(\sigma,\tau)\in C_{\bm{p}}$. 
	\end{lemma}
	
	\begin{proof}
		Let $(\sigma,\tau)\in C_{\bm{p}}$. Suppose that $R$ is a well-founded bigness witness above $(\sigma,\tau)$. We show that if $s\in R$ is not a leaf of~$R$, and the last entry of~$s$ is in $C_{\bm{p}}$, then~$s$ has an extension in~$R$ whose last entry is also in~$C_{\bm{p}}$. In this way we can keep extending until we get a leaf of~$R$ with last entry in~$C_{\bm{p}}$.

		Let~$s\in R$ be a non-leaf; let $(\rho,\nu)$ be the last entry of~$s$. Let 
		 \[
		 A = \left\{ \rho'\succ\rho \,:\,  (\exists^\infty n)\,\,\concat{s}{\smallseq{(\rho',\concat{\nu}{n})}}\in R \right\};
		 \]
		 so $A$ is dense above some $\rho_0\succeq \rho$. Since $\dom \fun{p}$ is comeagre in~$[\mindom{p}]$, and so in~$[\rho_0]$, we can find some $\rho'\in A$ with $|\fun{p}(\rho')|>|\nu|$; we then take some $n> \fun{p}(\rho')(|\nu|)$ such that $\concat{s}\smallseq{(\rho',\concat{\nu}{n})}\in R$. If $(\rho,\nu)\in C_{\bm{p}}$ then $(\rho',\concat{\nu}{n})\in C_{\bm{p}}$ as well. 
	\end{proof}

	\begin{corollary} \label{cor:C_p_is_very_big}
		Let $\bm{p}\in \iter$, and suppose that $(\sigma,\tau)\in E_{\bm{p}}$. Then for all $A\subseteq \upairs$, if $A$ is $\omega$-big above $(\sigma,\tau)$ then $A\cap E_{\bm{p}}$ is $\omega$-big above $(\sigma,\tau)$.
	\end{corollary}

	\begin{proof}
		The fact that $(\sigma,\tau)\notin \cl(\bad{p})$ implies that $\cl(\bad{p})$ is $\omega$-small above $(\sigma,\tau)$. By \cref{lem:C_p_is_quite_big} and  \cref{prop:closure_operator_for_pairs}, $\cl(\bad{p})\cup (\upairs \setminus C_{\bm{p}})$ is $\omega$-small above $(\sigma,\tau)$. If $A\cap E_{\bm{p}}$ is $\omega$-small above $(\sigma,\tau)$ then 
		\[
			A  = (A\cap \cl{(\bad{p})}) \cup (A\setminus C_{\bm{p}})\cup (A\cap E_{\bm{p}})
		\]
		is $\omega$-small above $(\sigma,\tau)$. 
	\end{proof}

	\begin{lemma} \label{lem:iterated_forcing:characterisation_of_C_p}
		Let $\bm{p}\in \iter$. The following are equivalent for $(\sigma,\tau)\in \upairs$: 
		\begin{enumerate}
			\item $(\sigma,\tau,\fun{p},\bad{p})\in \iter$ and extends $\bm{p}$.
			\item There is some extension $\bm{q}\le \bm{p}$ such that $(\sigma,\tau) = (\mindom{q},\stm{q})$. 
			\item $(\sigma,\tau)\in E_{\bm{p}}$. 
		\end{enumerate}
	\end{lemma}

	\begin{proof}
		(2)$\Rightarrow$(3): suppose that $\bm{q}\le \bm{p}$. As in the proof of \cref{lem:Hechler_with_bad_sets:possible_finite_extensions}, because $\bad{p}\subseteq\bad{q}$ we get $(\mindom{q},\stm{q})\notin \cl(\bad p)$. Also, $\mindom{q}\succeq \mindom{p}$ so $\fun{p}(\mindom{q})\le \fun{q}(\mindom{q}) \le \stm{q}$, so $(\mindom{q},\stm{q})\in C_{\bm{p}}$. 
	\end{proof}

\begin{lemma} \label{lem:iterated_forcing:density_of_long_stems}
	For all~$m$, the collection of $\bm{p}\in \iter$ for which $|\mindom{p}|,|\stm{p}|\ge m$ is dense in~$\iter$. 
\end{lemma}
	
\begin{proof}
	Let $\bm{p}\in \iter$. By \cref{lem:iterated_forcing:characterisation_of_C_p}, we need to show that there is a pair $(\sigma,\tau)\in E_{\bm{p}}$ with $|\sigma|,|\tau|\ge m$. Let $Q_m$ be the collection of $(\sigma,\tau)\succeq(\mindom{p},\stm{p})$ such that $|\sigma|,|\tau|\ge m$. Then~$Q_m$ is $\omega$-big above $(\mindom{p},\stm{p})$. By \cref{cor:C_p_is_very_big}, the collection of desirable $(\sigma,\tau)$ is $\omega$-big above $(\mindom{p},\stm{p})$, and so is nonempty. 
\end{proof}

For a filter $G\subset \iter$ we let
\[
	x_G = \bigcup \left\{ \mindom{p} \,:\,  p\in \iter \right\}
\]
and
\[
	y_G = \bigcup \left\{ \stm{p} \,:\,  p\in \iter \right\}.
\]

\Cref{lem:iterated_forcing:density_of_long_stems} implies that if $G$ is sufficiently generic, then 
\[
	(x_G,y_G)\in \cantor\times \baire. 
\]

For $\bm{p}\in \iter$ we let $[\bm{p}]$ be the collection of $(x,y)\in \cantor\times \baire$ such that: 
\begin{itemize}
	\item $\mindom{p}\prec x$ and $\stm{p}\prec y$;
	\item For all $m \geq \stm{p}$, $\fun{p}(x,m)\le y(m)$; and
	\item $(x,y)\notin \open{\bad{p}}$. 
\end{itemize}

This is a closed subset of $\cantor\times \baire$. 

\begin{lemma} \label{lem:iterated_forcing:strong_forcing_implies_forcing}
	For all $\bm{p}\in \iter$, $\bm{p}\force_{\iter} (x_G,y_G)\in [\bm{p}]$. 
\end{lemma}

\begin{proof}
	Suppose that $\bm{p}\in G$ and~$G$ is sufficiently generic. Let $\sigma\prec x_G$ and $\tau\prec y_G$. There is some $\bm{q}\in G$ extending~$\bm{p}$ such that $(\sigma,\tau)\preceq (\mindom{q},\stm{q})$; and $(\mindom{q},\stm{q})\in C_{\bm{p}}\setminus \bad{p}$. 
\end{proof}

Once again, we say that $\bm{p} \in \iter$ strongly forces $\varphi$ if $\varphi(x)$ holds for every $x \in [\bm{p}]$.

The proof of \cref{lem:Hechler:forcing_open} (using \cref{cor:C_p_is_very_big}) gives:

\begin{lemma} \label{lem:iterated:forcing_open}
	Let $A\in \ideal{I}$ be a name of an open subset of $\cantor\times \baire$, and let $\bm{p}\in \iter$. 
	\begin{enumerate}
		\item If $E_{\bm{p}}\subseteq \cl(A)$ then $\bm{p}\force (x_G,y_G)\in \open{A}$. 
		\item If $E_{\bm{p}}\nsubseteq \cl(A)$ then~$\bm{p}$ has an extension which strongly forces that $(x_G,y_G)\notin \open{A}$. 
	\end{enumerate}
\end{lemma}

\subsection{The generic is strongly meagre engulfing}

\begin{lemma} \label{lem:Cohen_sits_inside_I}
	If $D\subseteq 2^{<\omega}$ is dense, then 
	\[
		\left\{ \bm{p}\in \iter  \,:\,  \mindom{p}\in D \right\}
	\]
	is dense in~$\iter$. 
\end{lemma}

\begin{proof}
	Let $\bm{p}\in \iter$. The set $D\times \omega^{<\omega}$ is $\omega$-big above $(\mindom{p},\stm{p})$. By \cref{cor:C_p_is_very_big}, there is some $(\sigma,\tau)\in D\times \omega^{<\omega}$  in $E_{\bm{p}}$; apply \cref{lem:iterated_forcing:characterisation_of_C_p}. 
\end{proof}

\begin{corollary} \label{cor:iterated:first_coordinate_is_Cohen}
	If $G\subset \iter$ is sufficiently generic, then $x_G$ is $\ideal{I}$-Cohen. 
\end{corollary}

\begin{lemma} \label{lem:iterated_forcing:the_Hechler_part_is_total}
	Every $\bm{p}\in \iter$ forces that $\fun{p}(x_G)$ is total, and that $y_G$ dominates $\fun{p}(x_G)$. 
\end{lemma}

\begin{proof}
	For all~$m$, the collection $A_m$ of $\sigma\in 2^{<\omega}$ with $|\fun{p}(\sigma)|\ge m$ is dense above $\mindom{p}$. By \cref{lem:Cohen_sits_inside_I}, densely below~$\bm{p}$ we can find conditions~$\bm{q}$ with $\mindom{q}\in A_m$. This shows that~$\bm{p}$ forces that $\fun{p}(x_G)$ is total. By \cref{lem:iterated_forcing:strong_forcing_implies_forcing}, $\bm{p}$ forces that for all $m\ge |\stm{p}|$ we have $y_G(m)\ge \fun{p}(x_G,m)$. 
\end{proof}

\begin{lemma} \label{lem:iterated_forcing:domination}
	If $G\subset \iter$ is sufficiently generic, then $y_G$ is $\ideal{I}[x_G]$-dominating. 
\end{lemma}

\begin{proof}
	Let $\Phi\in \ideal{I}$ be a name of a partial continuous function from $\cantor$ to~$\baire$. We first consider totality. Let~$D$ be the set of Cohen conditions $\sigma\in 2^{<\omega}$ that decide if $\Phi(x_G)$ is total or not. This set~$D$ of conditions is dense. By \cref{lem:Cohen_sits_inside_I}, there is some $\bm{p}\in G$ with $\mindom{p}\in D$. If $\mindom{p}$ forces that $\Phi(x_G)$ is not total then we are done. Suppose otherwise; then $\dom \Phi$ is comeagre in $[\mindom{p}]$. 

	We show that $\bm{p}$ forces that~$y_G$ dominates $\Phi(x_G)$. Let $\bm{q}$ be any extension of~$\bm{p}$. 
	By passing to an extension, we may assume that $|\Phi(\mindom{q})|\ge |\stm{q}|$. We define a name~$g$ of a partial continuous function by letting, for all $\sigma\succeq \mindom{q}$,

	 \[
	 |g(\sigma)| = \min \left\{ |\fun{q}(\sigma)|, |\Phi(\sigma)|    \right\},
	 \]
	 and for $m<|g(\sigma)|$,
	 \[
	 g(\sigma,m) = 
	 	\begin{cases*}
	 		\stm{q}(m), & if $m<|\stm{q}|$; and \\
	 		\max\{ \fun{q}(\sigma,m), \Phi(\sigma,m) \}\, & otherwise. 
	 	\end{cases*}
	 \]
	 Then~$g$ is monotone ($\sigma\preceq \sigma'$ implies $g(\sigma)\preceq g(\sigma')$). Since both $\dom \Phi$ and $\dom \fun{q}$ are comeagre in $[\stm{q}]$, so is $\dom g$. By our assumption that $|\Phi(\mindom{q})|\ge |\stm{q}|$, we have $|g(\mindom{q})|\ge |\stm{q}|$, and by definition, $\stm{q} \ge g(\mindom{q})$. Since $\stm{q} \ge \fun{q}(\mindom{q})$, for all~$\sigma$, $g(\sigma)\ge \fun{q}(\sigma)$. It follows that $\bm{r}= (\mindom{q},\stm{q},g,\bad{q})$  is a condition that extends~$\bm{q}$, and it strongly forces that $g(x_G) = \fun{r}(x_G)$ dominates $\Phi(x_G)$. By \cref{lem:iterated_forcing:the_Hechler_part_is_total}, $\bm{r}$ forces that $y_G$ dominates $\Phi(x_G)$.
\end{proof}

Now \cref{cor:iterated:first_coordinate_is_Cohen,lem:iterated_forcing:domination,lem:truss} imply:

\begin{corollary} \label{cor:iterated_forcing:in_SME}
	If $G\subset \iter$ is sufficiently generic, then $(x_G,y_G)\in \SME{\ideal{I}}$. 
\end{corollary}

\begin{remark} 
	To prove \cref{lem:iterated_forcing:domination} we can break up the 2-step iteration into its two steps. The first is essentially \cref{lem:Cohen_sits_inside_I}. For the second, fixing an $\ideal{I}$-Cohen real~$x$, we can show that $\hechlerb(\ideal{I}[x])$ is the collection of triples $\bm{p}(x) = (\stm{p},\fun{p}(x), \bad{p}(x))$ for $\bm{p}\in \iter(\ideal{I})$ (where $\bad{p}(x) = \left\{ \tau  \,:\,  (\exists \sigma\prec x)\,\,(\sigma,\tau)\in \bad{p} \right\}$). We can argue that if $G\subset \iter(\ideal{I})$ is sufficiently generic, then the collection of conditions $\bm{p}(x_G)$ for $\bm{p}\in G$ is a generic filter for $\hechlerb(\ideal{I}[x_G])$. We can then quote \cref{prop:Hechler_with_bad_sets:dominating} to show that $y_G$ is $\ideal{I}[x_G]$-dominating. This is perhaps thematically more tidy, but the direct proof of \cref{lem:iterated_forcing:domination} is sufficiently simple. 
\end{remark}

\subsection{Separating $\SNE{\ideal{I}}$ from $\SME{\ideal{I}}$}
	
What is left to show now is that if $G$ is sufficiently generic, then $(x_G, y_G)$ does not $\ideal{I}$-compute any $\ideal{I}$-random.
	
\begin{proposition}\label{thm:iterated_forcing:no_random}
	If $\ideal{I}$ is closed under $\lhyp$, and $G\subset \iter(\ideal{I})$ is sufficiently generic, then $\ideal{I}[x_G,y_G]$ contains no $\ideal{I}$-random real. 
\end{proposition}

\begin{proof}
 The construction of \cref{prop:Hechler:no_randoms} translates to the forcing~$\iter$ without the need for any new ideas. We let $\Phi\in \ideal{I}$ be a name for a partial continuous function from $\cantor\times \baire$ to~$\cantor$. For $m<\omega$ we let $A_m = \left\{ (\sigma,\tau)  \,:\,  |\Phi(\sigma,\tau)|\ge m \right\}$. Given $\bm{p}\in \iter$, we assume that it forces that $\Phi(x_G,y_G)$ is total. So by \cref{lem:iterated:forcing_open}, each $A_m$ is $\omega$-big above every $(\sigma,\tau)\in E_{\bm{p}}$. We again fix $X\in \ideal{I}$ that computes $\Phi$, $C_{\bm{p}}$ and $\bad{p}$. For $\nu\in 2^{<\omega}$ we let $V_\nu = \left\{ (\sigma,\tau)  \,:\,  \nu\preceq \Phi(\sigma,\tau) \right\}$. For each $(\sigma,\tau)\in C_{\bm{p}}$ we define $P_{\sigma,\tau}$ as above: it is $2^{<\omega}$ if $(\sigma,\tau)\in \cl(\bad{p})$, otherwise it is the collection of~$\nu$ such that $V_{\nu}$ is $\omega$-big above $(\sigma,\tau)$. Working in the smallest admissible set containing~$X$, we can uniformly enumerate the sets $P_{\sigma,\tau}$; again, they are each infinite, so at some $X$-computable stage we see that they are all infinite, and we can obtain the finite-weight $S\in \ideal{I}$ that has infinite intersection with each $P_{\sigma,\tau}$. This is a Solovay test in~$\ideal{I}$ and $\bm{p}$ forces that it captures $\Phi(x_G,y_G)$. 
\end{proof}

Together with \cref{cor:iterated_forcing:in_SME}, and the fact that strong null engulfing reals compute randoms, we obtain the promised second half of \cref{thm:main}:

\begin{theorem}
	If $\ideal{I}$ is closed under $\lhyp$ then $\SNE{\ideal{I}}\subsetneq \SME{\ideal{I}}$.
\end{theorem}

	
	\bibliographystyle{plain}

\end{document}